\documentclass{article}
\usepackage{amssymb}
\usepackage{mathrsfs}
\usepackage[utf8]{inputenc}
\usepackage{tikz-cd}
\usetikzlibrary{cd}
\usepackage{amsmath,amsfonts}
\newcommand{\maps}{:}

\newcommand{\im}{\operatorname{im}}

\newcommand{\ZZ}{\mathbb{Z}}

\newcommand{\FF}{\mathscr{F}}
\newcommand{\GG}{\mathscr{G}}
\newcommand{\bracket}[1]{\left \langle #1 \right\rangle}

\newcommand{\parens}[1]{\left( #1 \right)}
\newtheorem{lemma}{Lemma}[section]
\newtheorem{proposition}{Proposition}[section]
\newtheorem{theorem}{Theorem}
\newtheorem{lettertheorem}{Theorem}

\newtheorem{letterconjecture}{Conjecture}

\newtheorem{definition}{Definition}

\newtheorem{conjecture}{Conjecture}[section]
\newtheorem{example}{Example}[section]
\newtheorem{remark}{Remark}[section]
\newtheorem{corollary}{Corollary}[section]
\newenvironment{proof}{\vspace*{2ex}\noindent {\em Proof:}
}{\hfill $\diamond$ \\[2ex]}

\usepackage[colorinlistoftodos,prependcaption,textsize=scriptsize]{todonotes}

\makeatletter
\providecommand\@dotsep{5}
\renewcommand{\listoftodos}[1][\@todonotes@todolistname]{%
  \@starttoc{tdo}{#1}}
\makeatother

\title{Gabriel's Theorem for Infinite Quivers}
\author{Nathaniel Gallup and Stephen Sawin}
\date{\today}
\begin{document}

\maketitle

\begin{abstract}
We prove a version of Gabriel's theorem for (possibly infinite dimensional) representations of infinite quivers. More precisely, we show that the representation theory of  quiver $\Omega$ is of unique type (each dimension vector has at most one associated indecomposable) and infinite Krull-Schmidt (every, possibly infinite dimensional, representation is a direct sum of indecomposables) if and only if $\Omega$ is eventually outward and of generalized ADE Dynkin type ($A_n$, $D_n$, $E_6$, $E_7$, $E_8$, $A_\infty$, $A_{\infty, \infty}$, or $D_\infty$). Furthermore we define an analog of the Euler-Tits form on the space of eventually constant infinite roots and show that a quiver is of generalized ADE Dynkin type if and only if this form is positive definite. In this case the indecomposables are all locally finite-dimensional and eventually constant and correspond bijectively to the positive roots (i.e. those of length $1$). 
\end{abstract} 

\section{Introduction}

Gabriel's famous theorem (\cite{gabriel1972}) states that a quiver $\Omega$ is of finite-type (has finitely many isomorphism classes of indecomposable representations) if and only if it is of ADE Dynkin type (the underlying graph of $\Omega$ is one of the ADE Dykin diagrams $A_n, D_n, E_6, E_7, E_8$). In this case, the function which sends the isomorphism class of a representation $V$ of $\Omega$ to its dimension vector (the function which assigns to every vertex $i$ the dimension of $V_i$) gives a bijection between the isomorphism classes of indecomposable representations and the dimension vectors whose length is $1$ with respect to the Tits form $\bracket{n , n}_\Omega = \sum_{ i \in \Omega'_0} n_i^2 -  \sum_{a \in \Omega'_1}n_{s(a)}n_{t(a)}$. In particular this shows that the representation category of an ADE Dynkin quiver is of \emph{unique type} as well, meaning each dimension vector has at most one indecomposable associated to it. 

The category $\text{Reps}_\text{f}(\Omega)$ of finite-dimensional representations of any finite quiver $\Omega$ is clearly Krull-Schmidt (every object is uniquely a finite direct sum of indecomposables). However Ringel (\cite{ringel2016}) showed that infinite-dimensional representations of ADE Dynkin quivers can also uniquely be written as a (possibly infinite) direct sum of indecomposables (we then call the category $\text{Reps}(\Omega)$ of possibly infinite-dimensional representations of $\Omega$ \emph{infinite Krull-Schmidt}). 

The situation for infinite quivers is complicated by the fact that it is no longer obvious that even locally finite-dimensional representations are infinite Krull-Schmidt. However, in \cite{bautista-liu-paquette2011}, Bautista, Liu, and Paquette showed that this is the case for locally finite-dimensional representations of certain infinite quivers including an eventually outward $A_{\infty, \infty}$ quiver and identified the indecomposables as being in bijection with the connected full subquivers (as is true in the $A_n$ case as well). 

In \cite{gallup2022decompositions}, we simultaneously weakened the finiteness hypotheses of \cite{ringel2016} and \cite{bautista-liu-paquette2011}, and showed that  for an eventually outward $A_{\infty, \infty}$ quiver $\Omega$ $\text{Reps}(\Omega)$ is of unique type and  infinite Krull-Schmidt. We furthermore showed that again the indecomposables in this case are FLEI (meaning locally finite-dimensional and isomorphic on all but finitely many arrows) and in bijection with the connected full subquivers. We also gave an example of a representation of an $A_{\infty, \infty}$ quiver which is not eventually outward whose category of representations is not infinite Krull-Schmidt. In this paper, we prove the following analogous result for $D_\infty$. 

\begin{lettertheorem}
The representation theory $\text{Reps}(\Omega)$ of any eventually outward $D_\infty$ quiver $\Omega$ is of unique type and infinite Krull-Schmidt. Taking dimension vectors gives a bijection between the indecomposables of $\Omega$ (which are all FLEI) and the positive roots of $\Omega$ (by which we mean functions $n$ from $\Omega_0$ to $\mathbb{Z}_{\geq 0}$ which are constant along a cofinite set of arrows and have the property that the directed limit $\lim_{\Omega'} \bracket{n , n}_{\Omega'}$ taken over certain finite subquivers $\Omega'$ is equal to $1$). 
\end{lettertheorem}

We combine this with our $A_{\infty , \infty}$ result from \cite{gallup2022decompositions} to prove the following version of Gabriel's theorem for infinite quivers. 

\begin{lettertheorem}[Gabriel's Theorem for Infinite Quivers]\label{thm: gab intro}
the representation theory $\text{Reps}(\Omega)$ of a connected quiver $\Omega$ is of unique type infinite Krull-Schmidt if and only if $\Omega$ is of generalized ADE Dynkin type and eventually outward. In this case all indecomposables of $\Omega$ are FLEI and in bijection with the positive roots of $\Omega$.  
\end{lettertheorem}

We now give a summary of each section of the paper. In Section \ref{sec: background} we provide the necessary background on quiver representations. In Section \ref{sec: FLEI} we define FLEI representations and give an explicit description of the dimension vector of any FLEI indecomposable of a $D_\infty$ quiver. In Section \ref{sec: partial order} we define a partial order on the set $\operatorname{indec}_{\text{FLEI}}(\Omega)$  of indecomposable FLEI subrepresentations of any representation $V$ of any quiver $\Omega$ and show that in the case that $\Omega$ is of type $D_\infty$ and the arrows all point away from the trivlent vertex, this partial order is well-founded. In Section \ref{sec: reflection functors} we extend  the language of reflection functors from \cite{Bernstein_1973} for use in infinite quivers and use them to show that in fact the partial order on $\operatorname{indec}_{\text{FLEI}}(\Omega)$ is well-founded for $\Omega$ any eventually outward $D_\infty$ quiver.

In Section \ref{sec: poset filtration of subrepresentations} we define an order-preserving function $F$ from $\text{indec}_{\text{FLEI}}(\Omega)$ to the set $\text{sub}(V)$ of subrepresentations of $V$ (which is a poset under inclusion). We prove that $F$ always has an almost gradation, i.e a function $G: \text{indec}_{\text{FLEI}}(\Omega) \to \text{sub}(V)$ with the property that $F_\beta = G_\beta \oplus \sum_{\alpha < \beta} F_\alpha$, and such that $G_\beta$ is a direct sum of indecomposable FLEI representations in $\beta$. We then show that in fact when $\text{indec}_{\text{FLEI}}(\Omega)$ is well-founded, $V = \bigoplus_{\alpha \in \text{indec}_{\text{FLEI}}(\Omega)} G_\alpha$, which completes the proof that $\text{Rep}(\Omega)$ is infinite Krull-Schmidt in this case. This includes the case when $\Omega$ is an $A_\infty$, $A_{\infty, \infty}$, $D_\infty$, or finite ADE Dynkin quiver (we call this collection \emph{generalized ADE Dynkin quivers}). 

In Section \ref{sec: roots} we define the root space of any infinite quiver $\Omega$ as above and show that for favorable $\Omega$, the directed limit $\lim_{\Omega'} \bracket{n , n}_{\Omega'}$ taken over certain finite subquivers $\Omega'$ converges to an element of $\mathbb{Z} \cup \{ \pm \infty \}$ for all elements $n$ of the root space. Defining this limit to be the Tits form of $n$ on $\Omega$, we then prove that this form is positive definite if and only if each component of $\Omega$ is an  generalized ADE Dynkin quiver, and that in this case, eventually outward implies the indecomposables are in bijection with the positive roots (those $n$ with Tits form $1$). 

In Section \ref{sec: infininte gabs theorem} we combine the results of the previous sections to prove Theorem \ref{thm: gab intro}, and finally in Section \ref{sec: future work} we make the following conjecture about a possible strengthening of Theorem \ref{thm: gab intro} via the removal of the ``eventually outward'' condition when restricting to the category $\text{Rep}_{\text{lf}}(\Omega)$ of locally finite-dimensional representations. 

\begin{letterconjecture}[Locally Finite-Dimensional Infinite Gabriel's Thm.]
For a connected quiver $\Omega$ the category $\text{Rep}_{\text{lf}}(\Omega)$ is of unique type and is infinite Krull-Schmit if and only if it is of generalized ADE Dynkin type. In this case all indecomposables of $\Omega$ are FLEI and in one-to-one correspondence with positive roots.
\end{letterconjecture}

\section{Background} \label{sec: background}

A quiver $\Omega=(\Omega_0, \Omega_1)$ is a set $\Omega_0$ of \emph{vertices} and a set $\Omega_1$ of ordered pairs of vertices called \emph{arrows}. If $a=(i,j) \in \Omega_1$ we say $s(a)=i$ is the source and $t(a)=j$ is the target. We will name the following classes of quivers, and call them the generalized ADE Dynkin diagrams (this name will be justified in due course!), in each case edges indicating arrows that could go in either direction, the subscript $n$ is the number of vertices, and in the $E_n$ case $n$ can only be $6$, $7$, or $8$.  

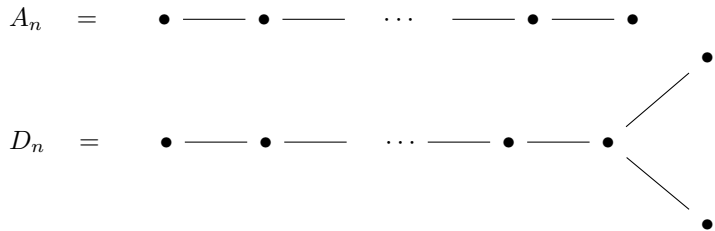
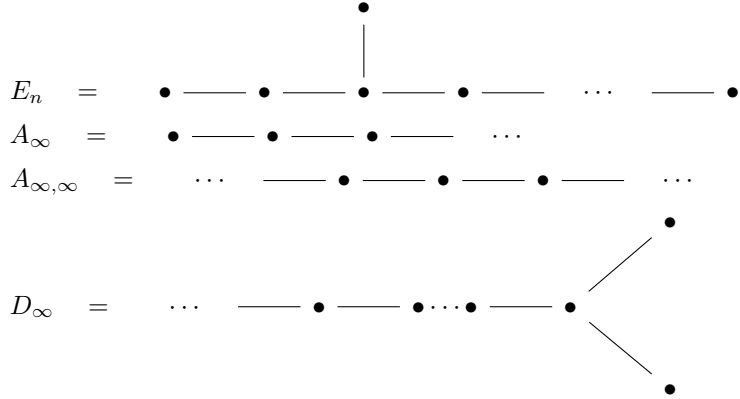
\begin{tikzcd}[arrows=-]
   A_n \quad = \qquad    \bullet \arrow[r] &   \bullet \arrow[r]  & \quad \cdots\quad \arrow[r]& \bullet \arrow[r]&  \bullet
 \end{tikzcd}
 
 \begin{tikzcd}[arrows=-]
 & & & & & \bullet \arrow[ld] \\
 D_n \quad = \qquad \bullet \arrow[r] & \bullet \arrow[r] &\quad \cdots \arrow[r] &\bullet \arrow[r] &
  \bullet  & \\
 & & & & &\bullet \arrow[lu]
\end{tikzcd}

 \begin{tikzcd}[arrows=-]
 & & \bullet \arrow[d] & & & \\
 E_n \quad = \qquad  \bullet \arrow[r] & \bullet \arrow[r] & \bullet \arrow[r] &  \bullet \arrow[r] & \quad \cdots \quad \arrow[r] &
  \bullet  
\end{tikzcd}

 \begin{tikzcd}[arrows=-]
   A_\infty\quad  = \qquad  \bullet \arrow[r] &   \bullet \arrow[r]  &  \bullet \arrow[r]  &\quad \cdots 
 \end{tikzcd}
 
 \begin{tikzcd}[arrows=-]
   A_{\infty,\infty} \quad  = \qquad  \cdots \quad \arrow[r] & \bullet \arrow[r]  &   \bullet \arrow[r] & \bullet \arrow[r] & \quad \cdots 
 \end{tikzcd}

 \begin{tikzcd}[arrows=-]
 & & & & \bullet \arrow[ld] \\
 D_\infty \quad = \qquad \cdots \quad \arrow[r] & \bullet \arrow[r] & \bullet \cdots \bullet \arrow[r] &
  \bullet  & \\
 & & & & \bullet \arrow[lu]
\end{tikzcd}

A vertex $i$ is a \emph{source} if it is the target of no arrow, and is a \emph{sink} if it is a source of no arrow.  

For $D_\infty$, we refer to the two univalent vertices on the right in the above diagram as \emph{leaves}, and the single trivalent vertex that is connected to the leaves as the \emph{trivalent} vertex.  The particular orientation in which  the leaves are the sole sources, there are no sinks, and there are  infinite infinite directed paths from the leaves is called the \emph{slide} orientation of $D_\infty$.

A \emph{walk} is a quiver morphism from $A_n$, $A_\infty$ or $A_{\infty, \infty}$ to $\Omega$, a \emph{trail} is a walk where the map on arrows is injective, and a \emph{path} is a trail where the map on vertices is injective.  If the domain of the path $p$ is a  $A_n$ or $A_\infty$ we call $p$, a \emph{journey}, and in this case we call the image of the left-most vertex the \emph{source} of the journey $s(p)$. When the domain is $A_n$ we will call the image of the right-most vertex the target $t(p)$ of the journey. A \emph{cycle} is trail with at least one arrow and with  source and target  the same.

We say that $\Omega$ is \emph{finitely-branching} if it is union of finitely many journeys.  It is called \emph{eventually outward} if every journey contains at most finitely-many arrows that point towards the source of those journeys.

A representation $V$ of $\Omega$ is an assignment $V(i)$ of a vector space (over a fixed field which will be denoted by $\mathbb{F}$) to each vertex $i \in \Omega_0$ and an assignment of a map $V(a)\maps V(s(a)) \to V(t(a))$ to each arrow $a \in \Omega_1$.  Vector spaces will not be assumed to be finite-dimensional unless specified.  A morphism $f \maps V \to W$ of representations is an assignment to each vertex $i$ of a linear map $f(i)\maps V(i) \to W(i)$ so that for each arrow $a$ there is a commutative diagram with $V(a)$, $W(a)$, $f(s(a))$ and $f(t(a))$. We say that $f$ is \emph{nontrivial} if it is not the zero morphism and not an isomorphism. 

We say that a representation $V$ of $\Omega$ is \emph{thin} if every vector space $V_i$ is $0$- or $1$-dimensional.

Recall that a category is \emph{Krull-Schmidt} if every object can be written as a finite direct sum of indecomposable objects, and given two such decompositions there is a bijection between the two sets of summands which takes each indecomposable   to an isomorphic indecomposable. We will say that a category is \emph{infinite Krull-Schmidt} if every object can be written as a possibly  infinite direct sum of indecomposable objects, and given two such decompositions there is a bijection between the two sets of summands which takes each indecomposable   to an isomorphic indecomposable. Note that if the object is \emph{isotypic}, that is to say all the indecomposables in the summand are isomorphic, then the bijection exists as long as the cardinalities are the same.

\section{FLEI Representations} \label{sec: FLEI}
\begin{definition}
A representation $V$ of a quiver $\Omega$ is called \emph{\textbf{f}inite \textbf{l}ocally \textbf{e}ventually \textbf{i}somorphic} or \emph{FLEI} if for each vertex $i$ in $\Omega$, $V(i)$ is finite-dimensional and there are only finitely many arrows $a$ in $\Omega$ for which $V(a)$ is not an isomorphism. More generally, if $\Omega'$ is a full subquiver of $\Omega$ and for every arrow $a$ of $\Omega$ not in $\Omega'$ we have $V(a)$ is an isomorphism, we say that $V$ is \emph{FLEI supported on $\Omega'$.} 
\end{definition}

The category of FLEI representations of $\Omega$, denoted $\operatorname{rep}_{\text{FLEI}}(\Omega)$, is the full subcategory of the category of representations of $\Omega$ whose objects are FLEI, and $\operatorname{rep}_{\text{FLEI}}(\Omega,\Omega')$ is the subcategory of those supported on $\Omega'$. Write $\operatorname{indec}_{\text{FLEI}}(\Omega)$ for the set of equivalence classes of indecomposable objects in $\operatorname{rep}_{\text{FLEI}}(\Omega)$.
\begin{definition}
  If $\Omega$ is a quiver and $\Omega'$ is a full subquiver, we say that $\Omega'$ is a \emph{retraction} of $\Omega$ if each component of $\Omega$ contains exactly one component of $\Omega'$ and every  cycle  of $\Omega$ is in $\Omega'$.
\end{definition}

\begin{remark}
  Notice that if $\Omega'$ is a retraction of $\Omega$, then for every vertex $i$ of $\Omega$ there is a unique vertex $j$ of $\Omega'$ (called the \emph{closest} vertex in $\Omega'$ to $i$) and a unique path from $i$ to $j$ with no arrows in $\Omega'$.  
\end{remark}

\begin{definition}
  If $\Omega$ is a quiver and $\Omega'$ is a retraction of it, define the functor
  \[\FF_{\Omega'}\maps \operatorname{rep}_{\text{FLEI}}(\Omega,\Omega') \to  \operatorname{rep}_{\text{FLEI}}(\Omega')\]
  which is just the restriction of each object and morphism to $\Omega'$.  Define the functor
  \[\GG_{\Omega'} \maps \operatorname{rep}_{\text{FLEI}}(\Omega') \to \operatorname{rep}_{\text{FLEI}}(\Omega,\Omega')\]
  as follows.  If $V$ is an object in   $\operatorname{rep}_{\text{FLEI}}(\Omega')$ then  $\GG_{\Omega'}(V)$ assigns to each vertex $i$  and arrow $a$ in $\Omega'$ the  space $V(i)$ or map $V(a)$ respectively. It assigns to each vertex $i$ \emph{not} in $\Omega'$ the space $V(j)$, where $j$ the closest vertex in $\Omega'$ to $i$.  It assigns to each arrow not in $\Omega'$ the identity (note its source and sink are assigned the same vector space).  Likewise, for a morphism $f$ in $\operatorname{rep}_{\text{FLEI}}(\Omega')$, define $\GG(f)$ to assign to each vertex $i$ in $\Omega'$ the map $f(i)$, and to each vertex $i$ not in $\Omega'$ the map $f(j)$ where $j$ is the vertex in $\Omega'$ closest to $i$.
\end{definition}

\begin{lemma} \label{lem:equivalence of categories}
 If  $\Omega'$ is a retraction of $\Omega$,  then $\FF_{\Omega'}$ and $\GG_{\Omega'}$ give an equivalence of categories between $\operatorname{rep}_{\text{FLEI}}(\Omega,\Omega')$ and $\operatorname{rep}_{\text{FLEI}}(\Omega')$.  
\end{lemma}

\begin{proof}
It is immediate that $\FF_{\Omega'}\circ \GG_{\Omega'}$ is the identity functor on $\operatorname{rep}_{\text{FLEI}}(\Omega')$. If $V$ is an object of $\operatorname{rep}_{\text{FLEI}}(\Omega,\Omega')$, then the natural isomorphism $f^V \maps \GG \circ \FF(V) \to V$ is given by $f(i)$ is the identity map if $i$ is in $\Omega'$ and if $i$ is not in $\Omega'$ and $j$ is the closest vertex in $\Omega'$ connected by a path to $i$, then $f(i)$ is the isomorphism $V(i)\to V(j)$ formed by composing $V(a)$ or its inverse for every arrow $a$ in that path. 
\end{proof}

\begin{lemma} \label{lem:indecomposables are the same } If $V \in \operatorname{indec}_{\text{FLEI}}(\Omega,\Omega')$ then  $V \in \operatorname{indec}_{\text{FLEI}}(\Omega)$. If $V \in \operatorname{indec}_{\text{FLEI}}(\Omega)$ then  $V \in \operatorname{indec}(\Omega)$.
\end{lemma}

\begin{proof}
    Suppose $V=W \oplus X$.  Then if $V$ is locally finite-dimensional so are $W$ and $X$.  If $V(a)$ is an isomorphism, then $W_a$ and $X_a$ must be.  Thus if $V$ is FLEI, so are $W$ and $X$, and if $V$ is supported in $\Omega'$ then so must $W$ and $X$ be. 
\end{proof}

\begin{corollary}\label{cor: flei indecomposables of D infinity}
    Every FLEI indecomposable representation of $D_\infty$ is obtained by extending an indecomposable representation of $D_n$ by isomorphisms along the long arm for some $n$. In particular the dimension vector is either thin or is of the following form.

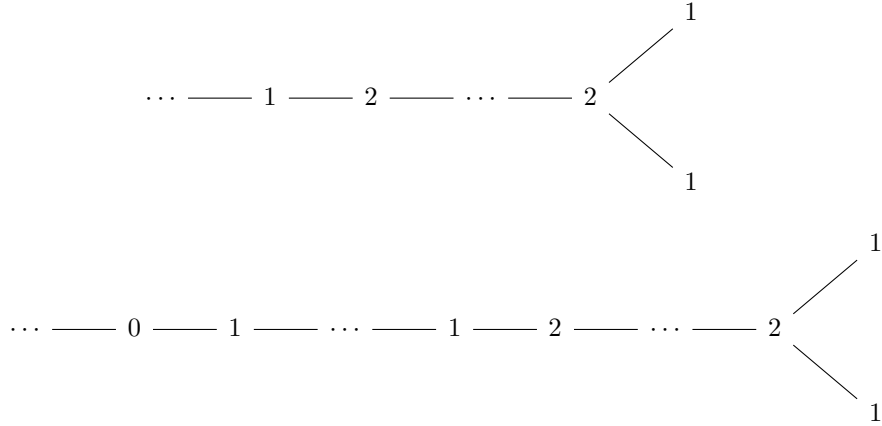
\begin{center}
\begin{tikzcd}[arrows=-]
 & & & & & 1 \arrow[ld] \\
 \cdots \arrow[r] & 1 \arrow[r] & 2 \arrow[r] & \cdots \arrow[r] & 2 &   & \\
 & & & && 1 \arrow[lu]
\end{tikzcd}
\end{center}

\begin{center}
\begin{tikzcd}[arrows=-]
 & & & & && & & 1 \arrow[ld] \\
 \cdots \arrow[r] & 0 \arrow[r] & 1 \arrow[r] & \cdots \arrow[r] & 1 \arrow[r] & 2 \arrow[r] & \cdots \arrow[r] & 2 &   & \\
 & & & &&& & & 1 \arrow[lu]
\end{tikzcd}
\end{center}

\end{corollary}
    
\begin{proof}
If $V \in \operatorname{indec}_{\text{FLEI}}(D_\infty)$, let $\Omega'$ denote a $D_n$-type subquiver of $D_\infty$ such that all arrows outside of $\Omega'$ are isomorphisms. Then by Lemma~\ref{lem:indecomposables are the same }, $\mathscr{F}_{\Omega'}(V) \in \operatorname{indec}(D_n)$, so by \cite{Humphreys72}[Section~12.1] it must be thin or of the second form above. Since $\mathscr{F}_{\Omega'}(V)$ and $V$ agree on $\Omega'$ and all arrows outside of $\Omega'$ are isomorphism, the result follows. 


\end{proof}

\begin{proposition} \label{prop: flei closed under sums and quotients and krull-schmidt}
Let $\Omega$ be a connected, eventually outward, finitely branched tree quiver. 
\begin{enumerate}
    \item  Every subrepresentation and quotient representation of a FLEI representation is FLEI. 
    
    \item Finite sums of FLEI representations are FLEI. 

    \item The category of FLEI representations is Krull-Schmidt.

\end{enumerate}

\end{proposition}

\begin{proof}
  \begin{enumerate}
  \item  Let $V$ be a FLEI representation of $\Omega$ and let $0 \to W \to V \to X \to 0$ be a subrepresentation and a quotient. At each vertex $i$ we have  $0 \to W(i) \to V(i) \to X(i) \to 0$ so clearly $W$ and $X$ are locally finite-dimensional. If $V(a)$ is an isomorphism then $W(a)$ is injective, and $X(a)$ is surjective,  so $W(a)$ and $X(a)$ are injective/surjective for all but finitely many $a$.  Since  $\Omega$ is a finitely branched tree it is a  union of journeys  $j_1, \ldots , j_n$.  Each journey $p_l$ contains finitely many arrows with $X(a)$ not surjective, and by eventually outward only finitely many arrows that point towards the start of the journey,  so each contains an infinite oriented path $p$ in which each arrow is surjective for $X$. Since $X(s(p))$ is finite-dimensional and the dimensions of the spaces assigned to vertices in $p$ are nonincreasing, all but finitely many arrows must leave the dimension constant, and therefore for all but finitely many arrows $a$ in $p$, $X(a)$ is an isomorphism, and hence $W(a)$ is an isomorphism.  So for all but finitely many arrows $a$ in $\Omega$, $W(a)$ and $X(a)$ are isomorphisms.  Hence $W$ and $X$ are FLEI.

    \item Suppose $V_1, \ldots, V_n$ are FLEI.  Clearly their sum is locally finite-dimensional.  On the other hand since there are only finitely many arrows that are nonisomorphic for each $V_m$, there are only finitely many for which any $V_m$ are nonisomorphic, so the sum is FLEI.
    
 \item If $V$ is FLEI, it is supported on some finite subquiver $\Omega'$.  We know that locally finite-dimensional representations on a finite quiver are Krull-Schmidt, so $\FF_{\Omega'}(V)$  can be written as a finite direct sum of indecomposables and hence by Lemma~\ref{lem:equivalence of categories} so can $V\in \operatorname{rep}_{\text{FLEI}}(\Omega,\Omega')$.  By Lemma~\ref{lem:indecomposables are the same } this writes $V$ as a finite direct sum of indecomposables in $\operatorname{rep}_{\text{FLEI}}(\Omega)$ and thus in $\operatorname{rep}(\Omega).$
\end{enumerate}

\end{proof}

\begin{example}
The hypotheses on $\Omega$ in Proposition \ref{prop: flei closed under sums and quotients and krull-schmidt} are all necessary, at least for the first part (we could relax connected to having finitely many components clearly). If $\Omega$ is $A_{\infty,\infty}$ oriented so that vertices alternate between sources and sinks (not eventually outward), then the representation with all vector spaces one dimensional and all maps isomorphism is FLEI, but the subrepresentation with all sinks having dimension $1$ and all sources having dimension $0$ is not, nor is the quotient by it.  If $\Omega$ is a union of infinitely many $A_2$ copies (not connected), the representation with all vector spaces having dimension $1$ and all maps being isomorphisms is FLEI, but again but the subrepresentation with all sinks having dimension $1$ and all sources having dimension $0$ is not, nor is the quotient by it.  Finally, if $\Omega$ has infinitely many vertices all connected by an arrow to a single sink (not finitely branching), the representation  with all vector spaces having dimension $1$ and all maps being isomorphisms is FLEI, but the quotient by the subrepresentation with only the sink having positive dimension is not (reversing the arrows gives an example where the subrepresentation is not FLEI).

\end{example}


\section{The Partial Order of Indecomposable FLEI Representations} \label{sec: partial order}


\begin{definition}
    Let $\Omega$ be a connected quiver and define $\operatorname{indec}_\text{FLEI}(\Omega)$ to be the set of isomorphism classes of indecomposable FLEI representations of $\Omega$. We denote such a class by $\alpha$ and we write $X \in \alpha$ if $X$ is an element of this isomorphism class.
\end{definition}

\begin{definition}\label{def: partial order on indecomposables}
  We define a relation on $\operatorname{indec}_\text{FLEI}(\Omega)$ by setting $\alpha \geq \beta$ if and only if there exists a sequence $\alpha = \gamma_1, \gamma_2, \ldots, \gamma_n = \beta$ in $\operatorname{indec}_\text{FLEI}(\Omega)$, representatives $W^{\gamma_i} \in \gamma_i$, and maps $W^{\gamma_1} \to W^{\gamma_2} \to \ldots \to W^{\gamma_n}$ which are all nontrivial: i.e. nonzero and nonisomorphic. 
\end{definition}

\begin{remark}
  We will freely abuse notation by eliding the difference between the set of indecomposables $\alpha$ and an individual representative $V \in \alpha$, saying for instance that  $\alpha$ injects into $\beta$ when for any representations $V \in \alpha$ and $W \in  \beta$ there exists an injection $V \to W$.
\end{remark}

\begin{example}
Let $U$, $V$, and $W$ be indecomposable representations of the quiver $\bullet \leftarrow \bullet$ given in the diagram below (the map in $V$ is an isomorphism). Then there are non-trivial maps $U \to V$ and $V \to W$, but no non-trivial maps $U \to W$ exist. 
\begin{center}
\begin{tikzcd}
U \arrow[d] & 1  \arrow[d] & 0 \arrow[l] \arrow[d]   \\
V \arrow[d] & 1   \arrow[d] & 1 \arrow[l]  \arrow[d]  \\
W & 0  & 1 \arrow[l]   \\
\end{tikzcd}  
\end{center}
This shows that we cannot define our partial order simply by stating ``$\alpha \leq \beta$ if there exists a non-trivial map $\beta \to \alpha$'', we must take the transitive closure of this relation.
\end{example}

\begin{lemma} \label{lem: break up non-trivial maps into compositions of simple non-trivial maps}
    If $f : U \to W$ is a non-trivial (nonzero, non-isomorphic) map of indecomposable FLEI modules which is not surjective and not injective then there exists an indecomposable FLEI module $V$, a non-trivial surjective map $g : U \to V$, and a non-trivial injective map $h : V \to W$. 
\end{lemma}

\begin{proof} 
Since $\im f \subseteq W$ is a submodule of a FLEI module, it is FLEI by part (1) of Proposition \ref{prop: flei closed under sums and quotients and krull-schmidt}, hence by part (3) of Proposition \ref{prop: flei closed under sums and quotients and krull-schmidt}, $\im f$ is a direct sum $\bigoplus_{\alpha \in A} W_\alpha$ of indecomposable FLEI modules. Denote by $\pi_\alpha : \im f \to W_\alpha$ the projection map. Then because $f$ is nonzero, there is some $\alpha \in A$ such that $\pi_\alpha \circ f \neq 0$, and hence we let $V = W_\alpha$ and $g= \pi_\alpha \circ f : U \to W_\alpha$ and note that $g$ is surjective, nonzero, and also not an isomorphism because by hypothesis $f$ must have a kernel. Then let $h$ be the inclusion map $V = W_\alpha \to \im f$ composed with the inclusion map $\im f \to W$. We note that $h$ is nonzero (since $W_\alpha$ is nonzero), injective, and not surjective (because $\im f \neq W$ by hypothesis). 
\end{proof}

The following corollary is immediate. 

\begin{corollary}\label{cor: all inj and surj}
\
    \begin{enumerate}
    
    \item If $\alpha \geq \beta$ there exists a finite chain $\alpha = \alpha_0 \geq \alpha_1 \geq \ldots \geq \alpha_n = \beta$ such that for each $1 \leq i \leq n$, either $\alpha_i$ injects  into $\alpha_{i + 1}$ or surjects onto $\alpha_{i + 1}$.
    
    \item If $\alpha \geq \beta$, there exists $\gamma_1, \ldots, \gamma_n$ such that there are maps $\alpha \to \gamma_1 \to \ldots \to \gamma_n \to \beta$ which alternate between injective and surjective. 
    
    \end{enumerate}
\end{corollary}

In \cite{gallup2022decompositions} we defined a partial order $\leq_{\mathcal{C}}$ on the set $\mathcal{C}(\Omega)$ of connected subquivers of any eventually outward finitely branching tree quiver $\Omega$ by letting $s >_{\mathcal{C}} t$ if one can get from $s$ to $t$ by a sequence of the following moves:
\begin{enumerate}

    \item We say that $t$ is a \emph{reduction} of $s$ if there exists an arrow $a$ in $s$ such that $t$ is the portion of $s$ behind $a$ (in terms of its orientation in $\Omega$). Thus $a$ is a boundary arrow of $t$ pointing away.
    
    \item We say that $t$ is an \emph{enhancement} of $s$ if there exists an arrow $a$ in $t$ such that $s$ is the portion of $t$ in front of $a$. Thus $a$ is a boundary arrow of $s$ pointing towards it.

\end{enumerate}

If $\Omega$ is of type $A_{\infty, \infty}$, by Proposition 7.1 of \cite{gallup2022decompositions} there is a bijection between the indecomposable FLEI representations of $\Omega$ and the set of connected components of $\Omega$ given by taking the support of a representation (the following lemma shows that the support of an indecomposable is always connected). 

\begin{lemma}\label{lem: support of indecomposable is connected}
If $\Omega$ is a tree and $V$ is an indecomposable representation of $\Omega$, the support of $V$ must be connected. 
\end{lemma}

\begin{proof}
If $\Omega$ is a tree and the support of $W$ is disconnected, then $W(i) = 0$ for some vertex $i$ of $\Omega$. In this case, $W = \bigoplus_{j} W^j$ where $j$ ranges over all vertices which are neighbors of $i$, and for each such $j$, $W^j$ is the subrepresentation of $W$ which agrees with $W$ for all vertices connected to $j$ not through $i$ (this is well-defined because $\Omega$ is a tree), and is zero elsewhere. 
\end{proof}

Then we have the relation $\leq$ on $\operatorname{indec}_\text{FLEI}(\Omega)$ and we have the partial order $\leq_\mathcal{C}$ on $\mathcal{C}(\Omega)$. We now show that these are in fact the same relation:

\begin{lemma}\label{lem: the two relations on type A are the same}
       Let $\Omega$ a finitely branching tree quiver. If $\alpha$ and $\beta$ represent thin indecomposable (necessarily FLEI) representations of $\Omega$ with supports $s_\alpha$ and $s_\beta$ then $\alpha > \beta$ if and only if $s_\alpha >_{\mathcal{C}} s_\beta$.  
\end{lemma}

\begin{proof}
If $s_\beta$ is obtained from $s_\alpha$ by a reduction then there is a surjective map $\alpha \to \beta$ and if $s_\beta$ is obtained from $s_\alpha$ by an enhancement then there is an injection $\alpha \to \beta$. Hence $s_\alpha >_{\mathcal{C}} s_\beta$ implies $\alpha > \beta$. 


Conversely, if there is a nontrivial surjective map $\alpha \to \beta$ then $s_\beta$ is a strict subset of $s_\alpha$ and any boundary arrows of $s_\beta$ which are in $s_\alpha$ must point away from $s_\beta$. Since $\Omega$ is a finitely branching tree quiver, only finitely many such boundary arrows can exist, hence $s_\beta$ can be obtained from $s_\alpha$ by a finite sequence of reductions. Similarly if there is a nontrivial injective map $\alpha \to \beta$ then $s_\beta$ can be obtained from $s_\alpha$ by a nontrivial sequence of enhancements. Hence $\alpha > \beta$ implies $s_\alpha >_{\mathcal{C}} s_\beta$. 
\end{proof}

The goal of the rest of this section is to prove that the relation $\leq$ on $\operatorname{indec}_\text{FLEI}(\Omega)$ is a well-founded partial order in the case when $\Omega$ is the mountain $D_\infty$ quiver (i.e. with all arrows pointed outward from the two leaves) shown in Figure \ref{fig: mountain}.

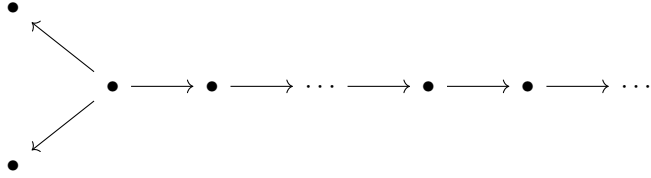
\begin{figure}
\begin{center}
 \begin{tikzcd}
\bullet  & & & &   \\
& \bullet \arrow[lu] \arrow[ld] \arrow[r] & \bullet \arrow[r] & \cdots \arrow[r] & \bullet \arrow[r] & \bullet \arrow[r]  & \cdots \\
\bullet & & & &  
\end{tikzcd}
\end{center}
\caption{The Mountain $D_\infty$ Quiver}
\label{fig: mountain}
\end{figure}

Next section we will use this to prove the same result for any eventually outward $D_\infty$ quiver.  It will first be helpful to adopt the convention that if $\alpha \in \operatorname{indec}_\text{FLEI}(\Omega)$ and $i$ is a vertex then $\alpha_i$ is the dimension of $W(i)$ where $W \in \operatorname{indec}_\text{FLEI}(\Omega)$ (necessarily independent of $W$).

We begin with a characterization of the types of maps in indecomposable representations of the $D_\infty$ mountain. 

\begin{lemma} \label{lem:  surjective or injective paths}
If  $V$ is an indecomposable representation of the $D_\infty$ mountain, and $p$ is a directed path, then the composition $V(p)$ of all maps in this path is either injective or surjective or both, and therefore of maximal rank. If the vector space on the trivalent vertex $k$ is two-dimensional, then the kernels of two otherwise non-intersecting paths of rank $1$ starting at $k$ have zero intersection. 
\end{lemma}

\begin{proof}
If $V$ is thin, i.e. all vector spaces are zero or one-dimensional, it suffices to prove that the map associated to every arrow between one-dimensional spaces is nonzero, which is a direct consequence of indecomposability. If $V$ is two-dimensional at vertex $1$, the same argument implies all the arrows get sent to maps of rank at least $1$.  If any arrow between a pair of two-dimensional spaces is rank $1$, we can split its image from its codomain, and carry that decomposition forward through each arrow to give a nontrivial decomposition, so all those arrows are rank $2$. 

Suppose the dimension of $V$ at the trivalent vertex $k$ is $2$ and that there are two paths $p$ and $q$ with $V(p)$ and $V(q)$ rank $1$, which do not overlap except at $k$ with kernels that intersect nontrivially and therefore are equal. Either this intersection is also the kernel of a rank $1$ path in the third direction, in which case it and any complement in $V(k)$ generate direct summands, of this intersection is independent from the kernel of said third path, in which case it and that kernel generate direct summands. See e.g. \cite{ringel2016} for a detailed analysis.

\end{proof}

Now we categorize the indecomposable representations of the $D_\infty$ mountain.

\begin{definition}
  Let $k$ be the trivalent vertex of $D_\infty$, and for $\alpha \in \text{indec}_\text{FLEI}(D_\infty)$ let $\alpha_i$ be the dimension of any representation in $\alpha$ at $i$. An indecomposable FLEI representation on the $D_\infty$ quiver is of \emph{type}
  \begin{itemize}
  \item[I] if $\alpha_k=0$
  \item[II] if $\alpha_k=1$ and $\alpha_j=1$ for all $j$ adjacent to $k$
  \item[III (a)] if $\alpha_k=1$ and $\alpha_j=1$ for exactly two $j$'s adjacent to $k$
  \item[III (b)] if $\alpha_k=2$
  \item[IV] if $\alpha_k=1$ and $\alpha_j=1$ for exactly one $j$ adjacent to $k$
  \item[V] if $\alpha_k=1$ and $\alpha_j=0$ for all $j$ adjacent to $k$
  \end{itemize}
  By Corollary \ref{cor: flei indecomposables of D infinity}, all indecomposable representation of $D_\infty$ are of one of these types.  
\end{definition}

\begin{lemma}\label{lm: ordering among types}
    If there is a nontrivial map from $\alpha$ to $\beta$ which is either injective or surjective, then $\alpha$ has type less than or equal to $\beta$. Therefore in an infinite chain of nontrivial maps between indecomposables, eventually all the indecomposables are of the same type.  
\end{lemma}

\begin{proof}
We consider the different possible types of $\beta$ as different cases. 

Case 1: Suppose that $\beta$ is of Type I. We claim that $\alpha$ must also be of Type I. If not then $\alpha_k \neq 0$. Since there is a nonzero map $\alpha \to \beta$, the supports of $\alpha$ and $\beta$ must intersect, therefore there exists some vertex $\ell$ where $\alpha_\ell \neq 0$ and $\beta_\ell \neq 0$. But the support of $\alpha$ is connected, so there is a directed path from $k$ to $\ell$ contained in the support of $\alpha$. Because $\beta_k = 0$ there must be an arrow $i \to j$ along this path with $\alpha_i , \alpha_j \neq 0$, $\beta_i=0$ and $\beta_j=1$. Hence we obtain the following diagram for representatives $V$ and $W$ of $\alpha$ and $\beta$ respectively.

\begin{equation}\label{eq: first diagram}
\begin{tikzcd}
V_j \neq 0 \arrow[d]  & V_i \neq 0 \arrow[l]  \arrow[d] & 
\\ W_j \cong \mathbb{F} & W_i = 0 \arrow[l] 
\end{tikzcd}
\end{equation}

Since the morphism $V \to W$ is either injective or surjective and $\beta_i = 0$, the map $V_i \to W_i$ must in fact be surjective. Since the map $V_i \to V_j$ is nonzero, this diagram cannot commute, contradiction. 

Case 2: Suppose that $\beta$ is of Type II. We claim that $\alpha$ cannot have $\alpha_k = 2$ and also cannot have $\alpha_j = 0$ for any $j$ adjacent to $k$ (this will show that $\alpha$ cannot be of Types III (a), III(b), IV, or V). The latter creates the following diagram

\begin{equation}\label{eq: one to zero bottom iso}
\begin{tikzcd}
V_j = 0 \arrow[d]  & V_i \cong \mathbb{F} \arrow[l]  \arrow[d] & 
\\ W_j \neq 0 & W_i \neq 0 \arrow[l, "\sim" '] 
\end{tikzcd}
\end{equation}

with $k = i$ which is impossible because the down arrows are both injections and the bottom arrow is an isomorphism hence $V_i \to W_i \to W_j$ is nonzero. In the former case, if $i$ and $j$ are the two vertices on the short arms, we obtain a diagram of the form 

\begin{equation}\label{eq: second diagram}
\begin{tikzcd}
V_i \cong \mathbb{F} \arrow[d]  & V_k \cong \mathbb{F}^2 \arrow[l]  \arrow[d] \arrow[r] & V_j \cong \mathbb{F} \arrow[d]  
\\ W_i \cong \mathbb{F} & W_k \cong \mathbb{F} \arrow[l] \arrow[r]  & W_j \cong \mathbb{F}
\end{tikzcd}
\end{equation}

where all maps are surjective and the outside and bottom maps are isomorphisms. Therefore the kernel of the map $V_k \to W_k$ must be equal to the kernel of the map $V_k \to V_j$ and to the kernel of the map $V_k \to V_i$. But these kernels are transverse lines in $\mathbb{F}$ and in particular are not equal so we obtain a contradiction. 

Case 3: Suppose that $\beta$ is of Type III (a), i.e. that $\beta_k=1$ and $\beta_j=1$ for exactly two $j$'s adjacent to $k$. If $\alpha$ has Type IV or V then this again creates Diagram \ref{eq: first diagram} with $V_i \to W_j \to W_j$ nonzero which is impossible.

Case 4: Now suppose that $\beta$ is of Type III (b), i.e. that $\beta_k = 2$. If $\alpha$ is of Type IV or V then there are two vertices $i$ and $j$ adjacent to $k$ with $a_i = a_j = 0$. Hence we obtain a diagram of the following form. 

\begin{equation}\label{eq: two zeros mapping to thick}
\begin{tikzcd}
V_i = 0 \arrow[d]  & V_k \cong \mathbb{F} \arrow[l]  \arrow[d] \arrow[r] & V_j = 0 \arrow[d]  
\\ W_i \cong \mathbb{F} & W_k \cong \mathbb{F}^2 \arrow[l] \arrow[r]  & W_j \cong \mathbb{F}
\end{tikzcd}
\end{equation}

Hence the image of the map $V_k \to W_k$ must be equal both to the kernel of the map $W_k \to W_i$ and to that of $W_k \to W_j$. Again these kernels are in fact transverse lines so this is impossible.  

Case 5: Suppose that $\beta$ is of Type IV. If $\alpha$ is of Type V, this again creates Diagram \ref{eq: first diagram} with $V_i \to W_j \to W_j$ nonzero which is impossible.


\end{proof}

If $\alpha$ is an isomorphism class of indecomposable representations of a $D_\infty$ quiver we define $T(\alpha)$ to be the set of vertices $i$ with $\alpha_i = 2$, and we define $L(\alpha)$ to be the set of vertices $i$ on the long arm with $\alpha_i \neq 0$. 

\begin{lemma}\label{lem: types iii a and iii b mixed}
\
\begin{enumerate}
    \item If $\alpha$ is of Type III(a) and $\beta$ is of Type III(b) and there is an injection or surjection $\alpha \to \beta$ then $L(\beta) \subseteq L(\alpha)$ and in particular $T(\beta) \subsetneq L(\alpha)$.  
    \item If $\alpha$ is of type III(b) and $\beta$ is of Type III(a) and there is an injection or surjection $\alpha \to \beta$ then $L(\beta) \subseteq T(\alpha)$ and in particular $L(\beta) \subsetneq L(\alpha)$.  
\end{enumerate}
\end{lemma}

\begin{proof}
\

\begin{enumerate}
\item If this is false, there must exist a diagram on the long arm of the following form

\begin{equation}\label{eq: one to zero two to one}
\begin{tikzcd}
V_j = 0 \arrow[d]  & V_i \cong \mathbb{F} \arrow[l]  \arrow[d] & 
\\ W_j \cong \mathbb{F} & W_i \cong \mathbb{F}^2 \arrow[l] 
\end{tikzcd}
\end{equation}

or Diagram \ref{eq: one to zero bottom iso}. We have explained why the latter is impossible. In the case of Diagram \ref{eq: one to zero two to one}, composing all maps along the long arm up to vertex $i$ yields Diagram \ref{eq: two zeros mapping to thick} which we have already shown is impossible. 

\item If this is false we obtain Diagram \ref{eq: second diagram} (where $V_k \to V_j$ and $W_k \to W_j$ are compositions of maps along the long arm) which we have already shown is impossible. 

\end{enumerate}
\end{proof}

\begin{proposition}\label{prop: no infinite sequence of non-trivial maps}
    If $\Omega$ is the $D_\infty$ mountain, there does not exist an infinite sequence $W_1, W_2, \ldots$ of indecomposable FLEI representations of $\Omega$ with nontrivial (nonzero and non-isomorphic) maps $W_1 \to W_2 \to \ldots$. 
\end{proposition} 

\begin{proof}
By Corollary \ref{cor: all inj and surj} we can assume that all maps are either injective or surjective and by Lemma~\ref{lm: ordering among types}, we can assume without loss of generality that all the indecomposables in such an infinite sequence are of the same type, i.e. all of type I, II, III, IV, or V.  If they are of any type \emph{except} III they are thin, and by Lemma~\ref{lem: the two relations on type A are the same}, they do not admit an infinite descending chain.

Finally suppose there is an infinite sequence of nontrivial morphisms between indecomposables of Type III. Suppose that $V \to W$ is a nontrivial map of such indecomposables. 

First note that if $V$ and $W$ are both of Type III(a) then $L(\beta) \subsetneq L(\alpha)$. Then if $V$ and $W$ are both of Type III(b) we claim that $V \to W$ must be a surjection. Indeed, if there is a nontrivial inclusion $V \to W$, then either $T(V) \subsetneq T(W)$, in which case we have the commutative square in Diagram \ref{eq: impossible square type 2}, or $T(V) = T(W)$ and then $L(V) \subsetneq L(W)$ in which case we have the commutative square in Diagram \ref{eq: impossible square type 1}.   

\begin{equation}\label{eq: impossible square type 2}
\begin{tikzcd}
V_i \cong \mathbb{F} \arrow[d]  & V_j \cong \mathbb{F}^2 \arrow[l]  \arrow[d] & 
\\ W_i \cong \mathbb{F}^2 & W_j \cong \mathbb{F}^2 \arrow[l] 
\end{tikzcd}
\end{equation}

\begin{equation}\label{eq: impossible square type 1}
\begin{tikzcd}
V_i = 0 \arrow[d] & V_j \cong \mathbb{F} \arrow[l]  \arrow[d] 
\\ W_i \cong \mathbb{F} & W_j \cong \mathbb{F} \arrow[l]
\end{tikzcd}
\end{equation}

In both cases $V_j \to W_j$ is an inclusion, hence by a dimension count it is an isomorphism, and $W_j \to W_i$ is also an isomorphism. Hence $V_j \to W_j \to W_i$ has rank $2$ which is at least one larger than the rank of $V_j \to V_i \to W_i$, contradiction.

Therefore $V \to W$ must be a surjection and so we have $T(W) \subseteq T(V)$, and $L(W) \subseteq L(V)$, with in one case the subset being proper. 

Suppose that $W_1 \to W_2 \to \ldots$ is a sequence of nontrivial Type III morphisms. Define $S(W_i)$ to be $T(W_i)$ if $W_i$ is of Type III(b) and $L(W_i)$ if $W_i$ is of Type III(a). Then we obtain two decreasing sequences of connected subsets of vertices on the long arm, $S(W_1) \supseteq S(W_2) \supseteq \ldots$ and $L(W_1) \supseteq L(W_2) \supseteq \ldots$ which, by the result above and Lemma \ref{lem: types iii a and iii b mixed}, has the property that for each $i$, either $S(W_i) \supsetneq S(W_{i + 1})$ or $L(W_i) \supsetneq L(W_{i + 1})$. Note that if $S(W_i)$ is infinite and strictly contains $S(W_{i + 1})$, then the latter must be finite, and similarly for $L(W_i)$. So at some point both sets must be finite, hence there can be only finitely many morphisms in this sequence. 
\end{proof}

\begin{corollary}\label{cor: the mountain is well-founded}
    If $\Omega$ is the $D_\infty$ mountain, the relation $\leq$ on $\text{FLEI}(\Omega)$ is a well-founded partial order.
\end{corollary}

\begin{proof}
Reflexivity and transitivity are obvious. For anti-symmetry, suppose that $V = W_1 , W_2, \ldots, W_n = V$ is a sequence of indecomposable FLEI modules and there exist nontrivial maps $W_1 \to \ldots \to W_n$. We can iterate this to obtain an infinite sequence $W_1, \ldots, W_n , W_2, \ldots, W_n, W_2, \ldots$ violating Proposition \ref{prop: no infinite sequence of non-trivial maps}. Therefore $\leq$ is a partial order. The same proposition clearly shows it is well-founded. 
\end{proof}


\section{Reflection Functors} \label{sec: reflection functors}

Let $\Omega$ be a quiver without cycles and let $i$ be a vertex. Suppose $i$ is either a source or a sink, and let $S_i(\Omega)$ be the same quiver with all arrows incident to $i$ reversed.  We defined a functor $\Phi_i$ from the (FLEI) representations of $\Omega$ to the (FLEI) representations of $S_i(\Omega)$ as follows.

On objects, let $V$ be a representation. define $\Phi_i(V)(j) = V(j)$ for all vertices $j \neq i$ and $\Phi_i(V)(a)=V(a)$ for all arrows not incident to $i$.  If $i$ is a source, define $F=\sum_{s(a)=i} V(a)$, a map from $V(i)$ to $X=\bigoplus_{j=t(a)\\s(a)=i} V(j)$ and define
\[\Phi_i(V)(i)= X/F[V_i]\]
and $\Phi_i(V)(a)$ for $a$ an arrow in $S_i(\Omega)$ with $t(a)=i$ is the obvious inclusion of  $V(j)$ to $X$ composed with this quotient.  A morphism $f$ from $V$ to $W$ gets mapped to $\Phi_i(f)$ with $\Phi_i(f)(j)=f(j)$ for $i \neq j$ and $\Phi_i(f)(i)$ the map
\[\bigoplus_{j=t(a)\\s(a)=i} V(j)/F[V_i] \to \bigoplus_{j=t(a)\\s(a)=i} W(j)/F[W_i]\]
induced by $f$.

If $i$ is a sink, define $\Phi_i$ in the same manner except define $F=\sum_{t(a)=i} V(a)$ and 
\[\Phi_i(V)(i)= F^{-1}[\{0\}].\]

Call the one taking a sink to a source $\Phi^+_i$ and the one taking a source to a sink $\Phi_i^-$.  The standard argument for the following result extends directly to the case of locally infinite representations on infinite quivers, but is  argued here for completeness. Let $L_i$ be the representation of $\Omega$ with a one dimensional vector space at $i$ and  $0$ dimensional vector spaces at all other vertices, and all arrows assigned the $0$ map.  Note this is irreducible.

\begin{lemma}\label{reflection-FLEI}
 When $\Omega$ is eventually outward The functors $\Phi_i$ takes FLEI representations to FLEI representations, and thus is a functor on $\operatorname{FLEI}(\Omega)$. 
\end{lemma}

\begin{proof}
  Obvious.
\end{proof}

\begin{lemma} \label{lm:reflection-indec}
  The functors $\Phi_i$ take indecomposables to indecomposables except $L_i$ is taken to the $0$ representation.  If the original indecomposable is not $L_i$ then the relevant sequence
  \begin{align*}
    0 \to \Phi^+_i(V)(i) &\to X \to V(i) \to 0\\
    0 \to V(i) &\to X \to \Phi_i^-(V)(i) \to 0\\
  \end{align*}
  is exact.
\end{lemma}

\begin{proof}
  We argue for $\Phi_i^+$, $\Phi_i^-$ works similarly.  Let $V$ be indecomposable and not equal to $L_i$,  so by Lemma~\ref{lm:representative}(b) $\sum_{t(a)=i}V(a)$ is onto.  This proves the exactness.   Suppose $\Phi_i^+(V)= W \oplus W'$.  Then for each $j \neq i$ one has $V(j) = W(j) \oplus W'(j)$ and for each arrow $a$ that does not have $i$ as a target $V(a)=W(a) + W'(a)$.  Thus $X=\bigoplus_{j=a(a)\\t(a)=i} V(j)$ decomposes into $X_W \oplus X_{W'}$, $\Phi_i^+(V)(i)$ decomposes into $Y(i) \oplus Y'(i)$, and thus the short exact sequence decomposes with $V= W(i) \oplus W'(i)$.
\end{proof}

\begin{lemma} \label{lm:representative}
 If $V$ is indecomposable and $i$ is a sink (resp. source) of $\Omega$ then $\sum_{t(a)=i} V(a)$ is surjective (resp. $\sum_{s(a)=i} V(a)$ is injective.
\end{lemma}

\begin{proof}
    Any complement of the image (resp. the kernel) is a direct summand.
\end{proof}

\begin{lemma} \label{lm:reflection-nontrivial}
  Suppose $F \maps V \to W$ is a nonzero morphism of indecomposables.  $\Phi_i(F)=0$  only if $F=0$ or $V \equiv L_i$.  $\Phi_i(F)$ is an isomorphism  only if $F$ is an isomorphism.
\end{lemma}

\begin{proof}
  For the second sentence $F$ is clearly $0$ if $C \equiv L_i$.  For each $j \neq i$ note that $\Phi_i(F)(j)= F(j)$ and therefore $\Phi_i(F)$ can only be $0$ if for every such $j$ $F(j)= 0$.  By Lemma~\ref{lm:representative}, if $V(i)$ is not $L_i$ then  $\sum_{t(a)=i} V(a)$ is onto, and since $F$ commutes with arrow maps $F(i)=0$.

  For the third sentence, suppose that $\Phi(F)$ and hence each $\Phi(F)(j)$, is an isomorphism, and again  for each $j \neq i$ we have immediately that $F(j)$ is an isomorphism.  Since then $\sum_{j=s(a)\\s(a)=i} F(j)$ is an isomorphism, $F(i)$ sits in a short exact sequence with $\Phi(F)(j)$ and $\sum_{j=s(a)\\s(a)=i} F(j)$ and this is an isomorphism.
  
\end{proof}

\begin{lemma} \label{lm:no-irred}
  Suppose $i$ is a sink of $\Omega$, suppose $V$ and $W$ are indecomposable representations of $\Omega$ and $f \maps V \to L_i$ and $g \maps L_i \to W$ are nontrivial morphisms.  Then $gf$ is a nontrivial morphism.  
\end{lemma}

\begin{proof}
  Clearly $f$ is surjective and if it is not an isomorphims must have a nontrivial kernel, and $g$ is injective and must have a notrivial cokernel, so $gf$ is nonzero and has nontrivial kernel and cokernel.
\end{proof}

\begin{proposition} \label{pr: reflection preserves well-founded}
  If $\Omega$ is a quiver and $i$ is a source or sink, and if the category of representations (or FLEI representations) of $S_i(\Omega)$ has well-founded indecomposables, so does $\Omega$.
\end{proposition}

\begin{proof}
  Let $\alpha_1 \to \alpha_2 \to \alpha_3 \to \cdot$ be a chain of morphisms of indecomposables, with each morphism neither the identity nor $0$.  If any of the $\alpha$ are isomorphic to $L_i$, then since $L_i$ is irreducible two such cannot be adjacent in the chain or the map between them would be the identity or $0$, so using Lemma~\ref{lm:no-irred}, we can remove each and replace the two maps involving it with their product, and the resulting chain will be infinite of the original one was.  So we may assume that the chain contains no representation isomorphic to $L_i$.  Applying the functor $\Phi_i^+$, we see the resulting chain will consist of indecomposable representations by Lemma~\ref{lm:reflection-indec}, and that each map is niether zero nor the identity by Lemma~\ref{lm:reflection-nontrivial}.  Since $\Omega'$ admits no infinite such chain, we conclude that $\Omega$ admits no such either.  
\end{proof}

\begin{proposition}
 Every eventually outward orientation of the $D_\infty$ quiver has well-founded indecomposables.
\end{proposition}

\begin{proof}
 It suffices by Proposition~\ref{pr: reflection preserves well-founded} together with Corollary~\ref{cor: the mountain is well-founded} to show that there is a sequence of reflections connecting the mountain orientation of every $D_\infty$ with any given eventually outward orientation. So suppose $\Omega$ is an eventually outward orientation  of $D_\infty$.  Because it is eventually outward   there must be a source, so let $i$ be the furthest source from the leaves of $D_\infty$, and assume $i$ is neither a leaf nor the trivalent vertex.  Consider a maximal oriented path from $i$ heading towards one or both leaves. This path necessarily ends in a sink. A reflection along that sink and then in turn along every vertex in the path results in an orientation with the furthest source closer than $i$ was to the leaves.  Inductively we will eventually reach an orientation where the furthest source is the trivalent vertex or a leaf.  It it is a leaf, then a reflection of one or both leaves results in the furthest source being the trivalent graph, which is the mountain orientation.   
\end{proof}


\section{Poset Filtration of Subrepresentations}
\label{sec: poset filtration of subrepresentations}
In this section the quiver is a tree (only necessary for  Lemma~\ref{lm: V is made of F}), but we assume that the FLEI indecomposables form a well-founded partial order under the ordering of Section~\ref{sec: FLEI}.

When we speak in the following of a (possibly infinite) sum of subrepresentations  of a quiver representation, we mean the subrepresentation that assigns to each vertex all finite linear combinations of vectors in individual subrepresentations at that vertex.  Recall a set $S$ of subrepresentations of some $V$  is independent if and only if for every finite subset $W \in S$, $W \cap \sum_{U \in S\setminus \{W\}} U = 0$. Thus if $S$ is independent $\sum_{U \in S} U$  is isomorphic to the direct sum $\bigoplus_{U \in S} U$.

\begin{remark}\label{rem: independence is finite}
 Independence is purely finite  in the  sense that if $S$ is a dependent set of subrepresentations (i.e. not independent), then there exists a finite subset $\{U_1, \ldots, U_n\} \subset S$ which is dependent, and indeed if we have written each $U_k= \sum_j W_{k,j}$ as a sum of subrepresentations,  we can choose for each $i$ a $U_i'$ which is a finite sum of the $W_{i,j}$ so that $\{U'_1, \ldots, U'_n\}$ is dependent. Also, since independence implies independence at each vertex, if $S$ is dependent and is a disjoint union of $T$ and $T'$, then either $T$, $T'$ or $\{\sum_{U \in T} U, \sum_{U \in T'} U\}$ is dependent.  
\end{remark}

  For the balance of this section assume all representations are subrepresentations of some fixed representation $V$ of $\Omega$.  So for instance below if $\beta$ is an isomorphism class of indecomposable FLEI representations then $\sum_{X \in \alpha<\beta } X$ is the sum of all suprepresentations of $V$ that are in some  isomorphism class $\alpha$ of indecomposables with $\alpha<\beta$.

\begin{definition}
    Let $V$ be a representation of $\Omega$. Define for each $\beta \in \operatorname{indec}_{\text{FLEI}}(\Omega)$, $F^\beta = \sum_{X \in \alpha \leq \beta} X$ and $F^{<\beta} = \sum_{\alpha < \beta} F^\alpha = \sum_{\alpha < \beta} X$. 
\end{definition}

The following proposition is then immediate.
\begin{proposition}
The map $\beta \mapsto F^\beta$ is order-preserving. 
\end{proposition}

\begin{lemma} \label{lem: nonzero maps between isomorphic indecomposables are isomorphisms}
    If $U , V \in \alpha$ are indecomposable FLEI representations of a quiver $\Omega$ and there is a nonzero map $f : U \to V$ then $f$ is an isomorphism. 
\end{lemma}

\begin{proof}
The image of $f$ is a FLEI submodule of $V$ by Proposition \ref{prop: flei closed under sums and quotients and krull-schmidt} which is nonzero by hypothesis, and is therefore a direct sum of nonzero indecomposable FLEI modules. Let $W \in \beta$ be one of these modules. Then the composition of $f$ with the projection onto $W$ is nonzero, hence $\alpha \geq \beta$ but the inclusion of $W$ into $V$ is also nonzero, hence $\beta \geq \alpha$. Therefore $\beta = \alpha$, so $W \subseteq V$ and yet $W \cong V$. Since $V$ is locally finite-dimensional, $W = V$ so the direct sum has one module in it, and $f$ is surjective. Again because $U$ and $V$ are locally finite-dimensional, this implies $f$ is an isomorphism. 
\end{proof}



\begin{lemma} \label{lm:summing beta reps}
  Let $V$ be a representation. 
  \begin{enumerate}
  \item If $U_1, \ldots, U_J \subseteq V$ are submodules with $U_j \in \beta_j$ then there exists an independent family of nonzero submodules $X_1 , \ldots X_I \subseteq V$ with $X_i \in \alpha_i$ such that $\sum_{j = 1}^J U_j = \sum_{i = 1}^I X_i$. For each $1 \leq i \leq I, 1 \leq j \leq J$ we denote by $f_{i,j} : U_j \to X_i$ the injection of $U_j$ into $\sum_{j = 1}^J U_j$ composed with the projection of $\sum_{i = 1}^I X_i$ onto $X_i$. Then if $f_{i , j}$ is nonzero, we have that $\alpha_i \leq \beta_j$ and if equality holds then $f_{i , j}$ is an isomorphism. Furthermore, for each $1 \leq i \leq I$, there exists $1 \leq j \leq J$ such that $\alpha_i \leq \beta_j$. 

    \item Suppose that $S \subset \beta$ (i.e. a set of subrepresentations of $V$ isomorphic to $\beta$) and $S \cup \{F_{<\beta}\}$ is independent.  If $X \in \alpha \leq \beta$ then either $X \subset \sum_{Z \in S} Z + F_{<\beta}$ or $S \cup \{X\} \cup \{F_{<\beta}\}$ is independent. In the latter case $X \in \beta$.


  \end{enumerate}
\end{lemma}

\begin{proof}

\begin{enumerate}
    \item By hypothesis $U_j$ is FLEI, hence by part (2) of Proposition \ref{prop: flei closed under sums and quotients and krull-schmidt} $\sum_{j = 1}^J U_j$ is also FLEI and by part (3) is therefore a finite direct sum of indecomposable FLEI submodules, say $\bigoplus_{i = 1}^I X_i$. If $f_{i,j} : U_j \to X_i$ is nonzero, then $\alpha_i \leq \beta_j$ by definition. If $f_{i,j} \neq 0$ and $\alpha_i = \beta_j$, then $f_{i , j}$ is an isomorphism by Lemma \ref{lem: nonzero maps between isomorphic indecomposables are isomorphisms}. The final claim holds because for each $1 \leq i \leq I$, $\sum_{j = 1}^J f_{i , j} : \sum_{j = 1}^J U_j \to X_i$ is the projection map which is nonzero because $X_i$ is, hence one of the $f_{i , j}$'s must also be nonzero. 

    \item Assume that $S \cup \{X\} \cup \{F^{<\beta}\}$ is dependent. Then by Remark~\ref{rem: independence is finite}, we can assume $S$ is finite,  and replace $F^{<\beta}$ with some $U$ which is a finite sum of elements of $\alpha<\beta$.   Then by  part (1) of this lemma, we can write  $\sum_{Z \in S} Z + X + U= Y_1 \oplus Y_2$ where $Y_1$ is a direct sum of elements of $\beta$ and $Y_2$ is a direct sum of elements of $\alpha<\beta$.  Again by part (1) of this lemma, the map of $U$ to $Y_1$ is $0$,  so the map of $\sum_{X \in S} X + Z$ to $Y_1$  surjective. Thus by local dimension count $Y_1$ must be a direct sum of at most $|S+1|$ subrepresentations isomorphic to $\beta$. If it is exactly $|S+1|$, then again by dimension count the map  $\sum_{X \in S} X + Z$ to $Y_1$  is a bijection. In that case $S \cup \{X\}$ is independent, so $\{U, \sum_{Z \in S} Z + X\}$ is dependent.  But the map from $U$ into $Y_1$ is $0$ and the map from $ \sum_{Z \in S} Z + X$ is injective, so they cannot have a nontrivial intersection, which gives a contradiction.  Therefore the sum is of fewer than $|S+1|$ subrepresentations. Then  since  $\sum_{Z \in S} Z$ does not intersect with $Y_2$  the map of $\sum_{Z \in S} Z$ into $Y_1$ is injective, hence by dimension count bijective. Then  one checks at each vertex that $Y_1 \subset \sum_{Z \in S} X + Y_2$ and thus $X \subset \sum_{Z \in S} Z + F^{<\beta}$.  
\end{enumerate}

\end{proof}

\begin{proposition} \label{pr:existence of G}
For each $\beta \in \text{FLEI}(\Omega)$ there is a subrepresentation $G_\beta \subseteq F_\beta$ such that $G_\beta$ is a direct sum of elements of $\beta$ and such that $F_\beta = F_{<\beta} \oplus G_\beta$. 
\end{proposition}

\begin{proof}

Let $\mathcal{B}$ be the set of subrepresentations of $V$ which are in $\beta$, and place a well-order on $\mathcal{B}$. We prove by transfinite induction the claim that for each $W \in \mathcal{B}$ there exists a subrepresentation $G_{W}$ of $F_\beta$, either equal to $W$ or to $0$, with the properties that the family $F_{< \beta}, G^U$ for $U \leq W$ (in this well-ordering of  $\mathcal{B}$) is independent and its span contains $W$. 

Suppose the claim is satisfied for all $U < W$. then $S=\{G_U \mid U<W\}$ and $W$ satisfy the assumptions of Lemma \ref{lm:summing beta reps} part (2), so either $W$ is contained in the sum $F_{<\beta} + \sum_U G_U$, in which case the claim is true with $G_W=0$, or $W=G_W$ makes $S \cup \{G_W\} \cup \{F_{<\beta}\}$ independent and again the claim is true.  

Now let $G_\beta = \sum_{W \in \beta} G_W$. On the one hand, the family $\{F_{< \beta}\} \cup \{ G_W \mid W \in \beta\}$ is independent. On the other hand, $F_\beta = \sum_{W \in \alpha \leq \beta} W = F_{<\beta} + \sum_{ W \in \beta} W \subset F_{<\beta} + \sum_{ W \in \beta} G_W$.  Therefore $F_\beta = F_{<\beta} \oplus G_\beta$ as desired.

\end{proof}

\begin{lemma} \label{lm: goddam G beta lemma}
If $\beta_1, \ldots, \beta_n$ are distinct isomorphism classes of FLEI modules then the $G_{\beta_i}$ are independent so $\sum_{i=1}^n G_{\beta_i} \cong \bigoplus_{i=1}^n G_{\beta_i}.$
\end{lemma}

\begin{proof}
  Suppose not. Since each $G_{\beta_i}$ is a direct sum of subrepresentations in $\beta_i$, there is a subrepresentation $H_{\beta_i} \subset G_{\beta_i}$ which is a \emph{finite} direct sum of subrepresentations in $\beta_i$ so that each  $H_{\beta_i}$ is independent of $F_{< \beta_i}$ and also $\{ H_{\beta_i}\}$ is dependent.  By Lemma~\ref{lm:summing beta reps}(1) write $\sum_{i=1}^n H_{\beta_i}$  as a direct sum of indecomposable FLEI modules. Then each $H_{\beta_i}$ embeds in a subset of this sum $Z_i \oplus Y_i$ where $Z_i$ is a direct sum of subrepresentations in $\beta_i$ and $Y_i$ is a direct sum of subrepresentations in $\alpha < \beta_i$.

Notice that $H_{\beta_i} \cap Y_i = 0$ because $H_{\beta_i}$ is independent of $Y_i \subset F_{<\beta_i}$.  Since the kernel of the map
 of $H_{\beta_i}$ into $Z_i$ is contained in $Y_i$, $H_{\beta_i}$ maps injectively  into $Z_i$.  By our supposition there is an $i_0$ such that $H_{\beta_{i_0}} \cap \sum_{i \neq i_0} H_{\beta_{i}}$.  Then projecting this intersection to $\bigoplus_i Z_i$ we see that it has a nonzero image in $Z_{i_0}$ and a nonzero image in $\bigoplus_{i \neq i_0} Z_i$, which is a contradiction.
\end{proof}

\begin{lemma} \label{lm:F is made of G}
For each $\beta$ we have $F_\beta = \sum_{\alpha \leq \beta} G_\alpha$.
\end{lemma}

\begin{proof}
Suppose this is not true. Then by well-foundedness there is some $\beta$ which is minimal among the set of indecomposable FLEI subrepresentations which do not satisfy the desired condition. Then by Lemma \ref{pr:existence of G} we have that $F_\beta = F_{<\beta} \oplus G_\beta$. But $F_{< \beta} = \sum_{\alpha < \beta} F_\alpha$ and by minimality of $\beta$ this is equal to $\sum_{\alpha < \beta} \sum_{\gamma \leq \alpha} G_\gamma$. Hence $F_\beta = \sum_{\alpha < \beta} G_\alpha + G_\beta = \sum_{\alpha \leq \beta} G_\alpha$, contradicting the assumption about of $\beta$.

\end{proof}

\begin{lemma} \label{lm: V is made of F}
If $V$ is a representation then $V = \sum_\beta F_\beta$.
\end{lemma}

\begin{proof}
  If not, there is a vertex $i$ of $\Omega$ so that $\sum_\beta F_\beta(i)$ is a proper subset of $V(i)$, which is to say there is an $x \in V(i)$ with $x$ not in $\sum_\beta F_\beta(i)$.  We build an indecomposable FLEI subrepresentation $W$ that contains $x$ and is in some $\beta$ recursively, so that each $W(j)$ has dimension $1$ or $0$.  First, let $W(i)$ be the span of $x$.  Then if $W(j)$ is one-dimensional and there is an arrow $a \maps j\to k$, then define $W(k)= V(a)[W(j)]$ which is either one-dimensional or $0$ dimensional.  Notice all the $1$ dimensional vertices are connected and the arrows connecting them are all assigned isomorphisms, so this representation is indecomposable, and obviously FLEI.  Then $W$ is in some $\beta$, contradicting the assumption.
\end{proof}

\begin{theorem}\label{thm: decomp of D infinity reps}
If $V$ is a representation of $\Omega$ and the poset of indecomposable FLEI subrepresentations of $V$ is well-founded,  then $V = \bigoplus_\alpha G_\alpha$. In particular every such representation can be written uniquely up to isomorphism and reordering, as a direct sum of indecomposable FLEI modules, and therefore the category of representations of a quiver with this property is infinite Krull-Schmidt.
\end{theorem}

\begin{proof}
  By Lemmas~\ref{lm: V is made of F} and~\ref{lm:F is made of G}, $V=\sum_\beta G_\beta$ and by Lemma~\ref{lm: goddam G beta lemma} $V = \bigoplus_\beta G_\beta$.  Since each $G_\beta$ is a direct sum of indecomposable FLEI, so is $V$.

  Suppose now that $V= \bigoplus_\beta H_\beta$ where each $H_\beta$ is a direct sum of subrepresentions in $\beta$.  Then by definition of $F_\beta$  each $H_\beta \subset F_\beta$, and the projection of $F_\beta$ onto $H_\alpha$ for $\alpha>\beta$ is zero by the definition of the ordering.  So for each $\beta$, $ F_\beta= \bigoplus_{\alpha \leq \beta} H_\alpha$ and therefore $H_\beta \cong F_\beta/ F_{< \beta} \cong  G_\beta.$  If $H_\beta$ is isomorphic to $G_\beta$, by dimension count at vertices they must have the same cardinality of indecomposable summands, so the uniqueness property of infinite Krull-Schmidt follows.
\end{proof}



\section{The Root Space of an Infinite Quiver}
\label{sec: roots}


\begin{definition}
    Define the \emph{root space} of $\Omega$ to be the set of functions $n : \Omega_0 \to \mathbb{Z}$ such that for all but finitely many arrows $a$, the difference in value between $s(a)$ and $t(a)$ is $0$ (we say then that $n$ is constant along $a$ in this case). 
  \end{definition}

\begin{definition}
  If $n$ is an element of the root space of a quiver $\Omega$ and constant on all arrows not in a subquiver $\Omega'$ we say that $n$ is \emph{anticonstant} on $\Omega'$ and that $\Omega'$ is $n$-anticonstant. Note that every such $n$ admits finite $n$-anticonstant subquivers and any subquiver containing an $n$-anticonstant subquiver is also $n$-anticonstant.
\end{definition}

\begin{definition}
The \emph{support} $\Omega^n$ of an element $n$ of the root space of $\Omega$ is defined to be the unique minimal full subquiver containing all vertices $i \in \Omega_0$ such that $n_i \neq 0$. 
\end{definition}

A subquiver $\Omega'$ of $\Omega$ is called \emph{IC} or ``injective on components''
(resp. \emph{BC} or ``bijective on components'') if no component of $\Omega$ contains multiple components of  $\Omega'$  (resp.  each component of $\Omega$ contains exactly  one component of $\Omega'$). 

\begin{lemma}\label{lem: finite subquivers are contained in finite IC subquivers}
If $\Omega'$ is a finite subquiver of $\Omega$ then there exists a finite IC subquiver $\Omega''$ of $\Omega$ containing $\Omega'$. 
\end{lemma}

\begin{proof}
Since $\Omega'$ is finite, it has finitely may components, so we can induct on the number of components. If $\Omega'$ has one component it is automatically IC. Suppose the result holds for finite subquivers with $p$ components. If $\Omega'$ has $p + 1$ components and is not IC, then there exist vertices $i, j \in \Omega_0'$ and a path from $i$ to $j$ in $\Omega$, but no path from $i$ to $j$ in $\Omega'$. Let $\Omega''$ be the union of $\Omega'$ and said path. Then $\Omega''$ has  $p$ or fewer components, so it is contained in a finite IC subquiver by the inductive hypothesis. But $\Omega''$ contains $\Omega'$ so the desired result follows. 
\end{proof}

Recall that a \emph{directed set} is a set $\mathcal{D}$ together with a reflexive, transitive relation $\leq$ on that set with upper bounds in that if  $a , b \in \mathcal{D}$ there exists $c \in \mathcal{D}$ such that $a \leq c$ and $b \leq c$.  If $\mathcal{C} \subset \mathcal{D}$ is \emph{cofinal} (for each $d \in \mathcal{D}$ there is a $c \in \mathcal{C}$ with $c \geq d$) then $\mathcal{C}$ is also directed.   If $\mathcal{D}$ is a directed set, $X$ is a topological space, and $f : \mathcal{D} \to X$ is a function, then we call $f$ a \emph{net} and  say $\lim_{d \in \mathcal{D}} f(d) = x$ if for all open neighborhoods $U$ of $x$, there exists $c \in \mathcal{D}$ such that if $d \geq c$ then $f(d) \in U$.   If $X$ is Hausdorff, then limits are unique when they exist. 

\begin{definition}
  If $(\mathcal{D},f)$ is a net, a subset $\mathcal{C} \subset \mathcal{D}$ is called \emph{stabilizing} if, for each $d \in \mathcal{D}$ there exists a $c \in \mathcal{C}$ with $c \geq d$ and $f(c)=f(d)$.  
\end{definition}

\begin{lemma} \label{lem: stabilizing preserves limits}
  If $\mathcal{C}$ is stabilizing for  $(\mathcal{D},f)$ then $\mathcal{C}$ is a directed set  and
  \[\lim_{d \in \mathcal{D}} f(d) = \lim_{c \in \mathcal{C}} f(c)\]
\end{lemma}

\begin{proof}
Note that $\mathcal{C}$ is cofinal and hence directed. Let $U$ be an open set in the codomain of $f$. If $d \in \mathcal{D}$ is such that for all $d' \in \mathcal{D}$ if $d' \geq d$, then  $f(d') \in U$, then choose $c \in \mathcal{C}$ with $c \geq d$ and it follows that if $c' \geq c$ and $c' \in \mathcal{C}$ then $f(c') \in U$.  On the other hand if $c \in \mathcal{C}$ is such that for all $c' \in \mathcal{C}$ with $c' \geq c$, $f(c') \in U$, then $c \in \mathcal{D}$, and if $d' \in \mathcal{D}$ with $d' \geq c$ choose $c' \geq d'$ with $c' \in \mathcal{C}$ and $f(d')=f(c')$, but $c' \geq c$ so $f(c') \in U$, which implies  $f(d') \in U$.
\end{proof}

Let $\mathcal{I}(\Omega)$ denote the set of finite IC subquivers of $\Omega$ together with the partial order given by inclusion. Given finite IC subquivers $\Omega'$ and $\Omega''$ of $\Omega$, it follows from Lemma \ref{lem: finite subquivers are contained in finite IC subquivers} that $\mathcal{I}(\Omega)$ is cofinal in the set of all finite subquivers and therefore is a directed set.

If $n$ and $m$ are two elements of the root space and  $\Omega'$ is a finite subquiver of $\Omega$, define the Euler form of $n$ and $m$ (respectively the Tits form of $n$) restricted to $\Omega'$ to be
\begin{align} \label{eq: tits form}
\bracket{n,m}_{\Omega'} &= \sum_{ i \in \Omega'_0} n_im_i - \frac{1}{2} \sum_{a \in \Omega'_1}\parens{ n_{s(a)}m_{t(a)} + n_{t(a)}m_{s(a)}}\\ \bracket{n,n}_{\Omega'} &= \sum_{ i \in \Omega'_0} n_i^2 -  \sum_{a \in \Omega'_1}n_{s(a)}n_{t(a)}  \nonumber
\end{align} 
where $n_i$ is the value of $n$ at $i$.  

 Any function $n : \Omega_0 \to \mathbb{Z}$ defines a function $\mathcal{I}(\Omega) \to \mathbb{Z}$ by $\Omega' \mapsto \bracket{n, n}_{\Omega'}$ which gives  a net in $\mathbb{Z} \cup \{ \pm \infty\}$ with the subspace topology (as a subspace of the extended real numbers).   

\begin{definition}\label{def: tits form limit}
  If $\Omega$ is a quiver and $n$ is an element of the root space, define
\begin{equation*}
\bracket{n,n}_{\Omega} = \lim_{\Omega' \in \mathcal{I}(\Omega)} \bracket{n,n}_{\Omega'}
\end{equation*}
\end{definition}

\begin{remark}
    If $\Omega$ is a finite quiver and $n$ is an element of the root space of $\Omega$, then $\lim_{\Omega' \in \mathcal{I}(\Omega)} \bracket{n , n}_{\Omega'} = \bracket{n,n}_{\Omega}$, so the notation ``$\bracket{n,n}_{\Omega}$'' introduced in Definition \ref{def: tits form limit} is not ambiguous.  
\end{remark}

\begin{remark}
  Since $\mathbb{Z} \cup \{ \pm \infty\}$ is Hausdorff this limit is unique if it exists.  If $\ell \in \ZZ$ then $\bracket{n,n}_{\Omega}= \ell$ if and only if there is an $\Omega'' \in\mathcal{I}(\Omega)$ such that $\bracket{n,n}_{\Omega'}= \ell$ for all finite IC subquivers $\Omega' \supseteq \Omega''$. 
\end{remark}

\begin{remark}
One could also consider taking the limit $\lim_{\Omega'} \bracket{n,n}_{\Omega'}$ over the directed set $\mathcal{G}$ of all finite (not necessarily IC) subquivers of $\Omega$. However this has no hope of converging. For example let $\Omega$ be any $A_\infty$ quiver with vertices $x_0, x_1, \ldots$ and arrows $a_1, a_2, \ldots$ and let $n$ denote the element of the root space defined by $n_{x_i} = 1$ for all $i \geq 0$. This should have Tits length $1$, and yet the limit $\lim_{\Omega' \in \mathcal{G}} \bracket{n,n}_{\Omega'}$ does not exist since any finite subquiver $\Omega'$ is contained in a finite subquiver $\Omega''$ (obtained by adding on $i$ vertices and no arrows, and is therefore \textbf{not IC}) such that $\bracket{n,n}_{\Omega''} = \bracket{n,n}_{\Omega'} + i$, hence the net never converges.
\end{remark}

\begin{lemma}\label{lem: tits form computed on support finite}
    Suppose $n$ is an element of the root space of a quiver $\Omega$ and $\Omega'$ is a finite subquiver of $\Omega$. Then $\bracket{n , n}_{\Omega' \cap \Omega^n} = \bracket{n , n}_{\Omega'}$. 
\end{lemma}

\begin{proof}
Every vertex in $\Omega'$ not in $\Omega^n$ clearly contributes only a zero term to the Tits form. Furthermore any arrow in $\Omega'$ not in $\Omega^n$ must have either its source or target not in $\Omega^n$ because $\Omega^n$ is a full subquiver by definition. Hence this also contributes only a zero term to the Tits form, so we obtain the desired equality. 
\end{proof}

We will use Lemma~\ref{lem: tits form computed on support finite} to show $\bracket{n , n}_\Omega = \bracket{n , n}_{\Omega^n}$. Define $\mathcal{J}_n(\Omega)$ to be the set of all finite subquivers $\Omega'$ whose intersection with $\Omega^n$ is IC in $\Omega^n$, which is clearly cofinal in the directed set of finite subquivers and hence a directed set by Lemma~\ref{lem: finite subquivers are contained in finite IC subquivers}.  Since $n$ is in the root space, there exists a finite $n$-anticonstant subquiver.  By Lemma~\ref{lem: finite subquivers are contained in finite IC subquivers} we can extend this to be IC, and again by Lemma~\ref{lem: finite subquivers are contained in finite IC subquivers} we can extend it to $\Omega^*$ so that its intersection with $\Omega^n$ is IC in $\Omega^n$ (it will still be IC in $\Omega$ because we did not add any new components). Define $\mathcal{I}'(\Omega)$ and $\mathcal{J}'_n(\Omega)$ to be the subsets of $\mathcal{I}(\Omega)$ and $\mathcal{J}_n(\Omega)$ respectively of subquivers containing $\Omega^*$. It is easy to see that the following equations hold (in fact they true for any net in any topological space).  
\begin{align}
  \label{eq:intial-segment}
   \lim_{\Omega' \in \mathcal{I}(\Omega)} \bracket{n,n}_{\Omega'}&= \lim_{\Omega' \in \mathcal{I}'(\Omega)} \bracket{n,n}_{\Omega'}\\
   \lim_{\Omega' \in \mathcal{J}_n(\Omega)} \bracket{n,n}_{\Omega'}&= \lim_{\Omega' \in \mathcal{J}'_n(\Omega)} \bracket{n,n}_{\Omega'} \nonumber
\end{align}

\begin{lemma}\label{lem: stabilizing subset}
  Given an element $n$ of the root space of a quiver $\Omega$, the set $\mathcal{I}'(\Omega)$ is a stabilizing subset of $\mathcal{J}'_n(\Omega)$ for the Tits form.
\end{lemma}

\begin{proof}
First we need to prove that $\mathcal{I}'(\Omega) \subset \mathcal{J}'_n(\Omega)$, that is if $\Omega'$ is IC and contains $\Omega^*$ then its intersection with $\Omega^n$ is IC in $\Omega^n$. Second we need to prove  that $\mathcal{I}'(\Omega)$ is stabilizing, that is if $\Omega'$ has an IC intersection with $\Omega^n$ and contains $\Omega^*$, it is contained in an $\Omega''$ that is IC and has intersection with $\Omega^n$ that is IC in $\Omega^n$ and agrees with $\Omega'$ on the Tits form.

For the first suppose it is not the case that $\Omega' \cap \Omega^n$ is IC in $\Omega^n$.  Then there exist $i,j \in \Omega' \cap \Omega^n$ with a path connecting them in $\Omega^n$ but no path connecting them in $\Omega' \cap \Omega^n$. In particular since $\Omega'$ is IC there is a path connecting them in $\Omega'$ but not one in $\Omega' \cap \Omega^n$.  We can assume that $i$ and $j$ are minimal with this property in that there is a path $p$  connecting them in $\Omega'$ with no vertices in the path in $\Omega^n$ except the endpoints. That means that $n_i$ and $n_j$ are nonzero but $n$ is equal to zero on every vertex in between on $p$.  That means that the arrows joining $i$ and $j$ to the path are not constant (if $p$ consisted of just one arrow connecting $i$ and $j$ then it would be in $\Omega' \cap \Omega^n$ since $\Omega^n$ is full) and therefore are in $\Omega^*$.  That means that $i$ and $j$ are in $\Omega^* \cap \Omega^n$ and since this is IC in $\Omega^n$ there is a path in $\Omega^* \cap \Omega^n$, which is a contradiction.

For the second, consider any two vertices $i,j$ in $\Omega'$ that are connected in $\Omega$ by a path that does not include vertices in $\Omega'$ other than the endpoints.  Thus no arrows in the path are in $\Omega^*$ and therefore $n$ is constant along this path. If the constant value is nonzero then the entire path is in $\Omega^n$ and so since $\Omega' \cap \Omega^n$ is IC as a subquiver of $\Omega^n$, there is a path connecting $i$ and $j$ in $\Omega' \cap \Omega^n$. Otherwise, all vertices in the path are not in $\Omega^n$ an we add all those vertices to $\Omega'$. We repeat this until no such vertices $i$ and $j$ remain, and this process terminates because each time we add a path, it decreases the number of components of $\Omega'$ which is finite. The result is $\Omega''$ which is IC and contains $\Omega^*$ and has the same intersection with $\Omega^n$ as $\Omega'$. Therefore $\Omega'' \in \mathcal{I}'(\Omega)$ and by Lemma~\ref{lem: tits form computed on support finite} $\bracket{n , n}_{\Omega''} = \bracket{n , n}_{\Omega'}$.
\end{proof}

\begin{proposition} \label{prop: limit computed by support}
For all $\Omega$ and $n$, $\bracket{n,  n}_\Omega = \bracket{n , n}_{\Omega^n}$, that is to say if one exists the other does and in that case their values are equal.
\end{proposition}

\begin{proof}
By Equation~\ref{eq:intial-segment} we have 
\[\bracket{n,n}_{\Omega} = \lim_{\Omega' \in \mathcal{I}(\Omega)} \bracket{n,n}_{\Omega'}=  \lim_{\Omega' \in \mathcal{I}'(\Omega)} \bracket{n,n}_{\Omega'}\]
By Lemmas~\ref{lem: stabilizing subset} and~\ref{lem: stabilizing preserves limits} we obtain
\[ \lim_{\Omega' \in \mathcal{I}'(\Omega)} \bracket{n,n}_{\Omega'}= \lim_{\Omega' \in \mathcal{J}_n'(\Omega)} \bracket{n,n}_{\Omega'}\]
and again by Equation~\ref{eq:intial-segment} it follows that
\[ \lim_{\Omega' \in \mathcal{J}_n'(\Omega)} \bracket{n,n}_{\Omega'}= \lim_{\Omega' \in \mathcal{J}_n(\Omega)} \bracket{n,n}_{\Omega'}\]
Thus $\bracket{n,n}_{\Omega} = \lim_{\Omega' \in \mathcal{J}_n(\Omega)} \bracket{n,n}_{\Omega'}$.
Finally, consider the map from $\mathcal{J}_n(\Omega)$ to $\mathcal{I}(\Omega^n)$, given by intersection.  It is order preserving, onto, by Lemma~\ref{lem: tits form computed on support finite} preserves the Tits form, and has the property that for any $\Omega' , \Omega'' \in \mathcal{I}(\Omega^n)$ and $\Lambda' \in \mathcal{J}_n(\Omega)$ with $\Omega' \subseteq \Omega''$ and $\Lambda' \cap \Omega^n = \Omega'$, there exists $\Lambda'' \in \mathcal{J}_n(\Omega)$ with $\Lambda' \subseteq \Lambda''$ and $\Lambda'' \cap \Omega^n = \Omega''$. Thus it follows that
  \[\lim_{\Omega' \in \mathcal{J}_n(\Omega)} \bracket{n,n}_{\Omega'}=\lim_{\Omega' \in \mathcal{I}(\Omega^n)} \bracket{n,n}_{\Omega'}\]
  and the proposition is proven.
\end{proof}

Now that we can calculate the Tits form on $\Omega^n$, we can determine when it converges to a finite or infinite value.

\begin{lemma}\label{ref: sequence of IC subquivers}
    Let $\Omega$ be a quiver and $\Omega' \subsetneq \Omega''$ finite IC subquivers. Then there exists a sequence of finite IC subquivers $\Omega' = \Lambda^0 \subsetneq \ldots \subsetneq \Lambda^\ell =\Omega''$ such that $\Lambda^{p + 1}$ can be obtained from $\Lambda^p$ by one of the following moves: (1) add a vertex not in $\Lambda^p$ and an arrow connecting it to a vertex that is in $\Lambda^p$, (2) add an arrow not in $\Lambda^p$ between two vertices which are in the same component of $\Lambda^p$, or (3) add a vertex in a component of $\Omega$ which does not contain any vertices from $\Lambda^p$. 

    \begin{enumerate}
        \item If $\Omega'$ is a BC subquiver, then move (3) cannot occur. 
        \item If $\Omega'$ contains all cycles of $\Omega$ then move (2) cannot occur. 
    \end{enumerate}
\end{lemma}

\begin{proof}
We proceed by induction on the total number of edges and vertices in $\Omega''$ not in $\Omega'$. If $\Omega''$ has an arrow not in $\Omega'$ whose removal does not create a new component (leaves it as IC), then it comes via a move (2) from an IC subquiver that inductively can be built from moves 1-3.  If $\Omega''$ contains a vertex with no arrows, it is necessarily in a distinct component of $\Omega$ and comes via a move (1) from an IC subquiver that inductively can be built from moves 1-3. Suppose that neither such a redundant arrow, nor an isolated vertex exists. Then $\Omega''$ must contain an arrow not in $\Omega'$, the removal of which separates $\Omega''$ into two pieces, one of which contains no elements of $\Omega'$ (since $\Omega'$ is IC) and therefore must be a tree (since otherwise an arrow of the first type above would exist). This implies that an arrow whose removal minimizes the number of arrows in that second piece will be a leaf whose removal leaves an isolated vertex, and therefore it comes via a move 1 from an IC subquiver that inductively can be built from moves 1-3.

\begin{enumerate}
    \item If $\Omega'$ is a BC subquiver, then any IC subquiver containing it is also BC and therefore there are no components of $\Omega$ which do not contain vertices from any $\Lambda^p$.
    
    \item If $\Omega'$ contains all cycles of $\Omega$, then no cycle of $\Omega''$ contains any arrow not in $\Omega'$, so it is not possible to add an arrow through move (2).  
\end{enumerate}
\end{proof}

\begin{lemma}\label{lem: constant on complement of chain of quivers}
Suppose $n$ is supported on $\Omega$ (i.e. $\Omega=\Omega^n$) and anticonstant on some finite IC subquiver $\Omega'$ which is contained in another finite IC subquiver $\Omega''$.   

\begin{enumerate}
    \item If $\Omega'$ is a BC subquiver then $\bracket{n , n}_{\Omega''} \leq \bracket{n , n}_{\Omega'}$ with the inequality strict if $\Omega''$ contains a cycle not in $\Omega'$. 
    \item If $\Omega'$ contains all cycles of $\Omega$ then $\bracket{n , n}_{\Omega''} \geq \bracket{n , n}_{\Omega'}$ with the inequality strict if $\Omega''$ has more components than $\Omega'$. 
\end{enumerate}
\end{lemma}

\begin{proof}
By Lemma \ref{ref: sequence of IC subquivers} there exists a sequence of finite IC subquivers $\Omega' = \Lambda^0 \subsetneq \ldots \subsetneq \Lambda^\ell =\Omega''$ such that $\Lambda^{p + 1}$ can be obtained from $\Lambda^p$ by a move of type (1), (2), or (3). The fact that $n$ is anticonstant on $\Omega'$ and is supported on $\Omega$ implies the following. If the move is type (1) then $\bracket{n , n}_{\Lambda^{p + 1}} = \bracket{n , n}_{\Lambda^{p}}$, if the move is type (2) then $\bracket{n , n}_{\Lambda^{p + 1}} < \bracket{n , n}_{\Lambda^{p}}$ and if the move is type (3) then $\bracket{n , n}_{\Lambda^{p + 1}} > \bracket{n , n}_{\Lambda^{p}}$.
    \begin{enumerate}
        \item In this case, move (3) cannot occur, and if there is a cycle in $\Omega''$ not in $\Omega'$ then move (2) must occur, so we obtain the desired inequalities. 
        \item In this case, move (2) cannot occur, and if $\Omega''$ has more components than $\Omega'$ then move (3) must occur, so again we obtain the desired inequalities. 
    \end{enumerate}
  \end{proof}

\begin{definition}
A quiver is said to have \emph{finite zeroth homology} if the rank of the zeroth homology of the underlying graph is finite, i.e.  it has only finitely many connected components and \emph{finite first homology} if the rank of the first homology of the underlying graph is finite, i.e. it has only finitely many cycles. 
\end{definition}

\begin{lemma}\label{lem: finite homology equals finite retraction}
    A quiver has finite zeroth and first homology if and only if it has a finite retraction. 
\end{lemma}

\begin{proof}
If $\Omega$ has a finite retraction it obviously has only finitely many cycles and connected components.  
If $\Omega$ has finite zeroth and first homology, then there is some finite subquiver which contains all cycles and contains a point from every connected component of $\Omega$. By Lemma \ref{lem: finite subquivers are contained in finite IC subquivers} this is contained in a finite IC subquiver which still contains a vertex from each connected component of $\Omega$ and is therefore BC as well. 
\end{proof}

\begin{corollary}\label{cor: constant on complement of retraction}
Suppose $n$ is supported on $\Omega$ and constant on all arrows not in some finite retraction $\Omega^*$ of $\Omega$ then $\bracket{n,  n}_\Omega = \bracket{n , n}_{\Omega^*}$.
\end{corollary}

\begin{proof}
A retraction is both BC and contains all cycles of $\Omega$, hence Lemma \ref{lem: constant on complement of chain of quivers} gives the desired equality. 
\end{proof}

\begin{proposition}\label{lpr: convergence of roots based on supports}
Let $\Omega$ be a quiver and $n$ an element of the root space of $\Omega$. 
\begin{enumerate}
    \item If $\Omega^n$ has finite zeroth and first homology, then there exists a finite IC subquiver $\Omega^*$ of $\Omega$ such that $\bracket{n , n}_{\Omega} = \bracket{n,n}_{\Omega^*}$. 
    
    \item If $\Omega^n$ has finite first homology but infinite zeroth homology then $\bracket{n , n}_{\Omega} = +\infty$.

    \item If $\Omega^n$ has finite zeroth homology but infinite first homology then $\bracket{n , n}_{\Omega} = -\infty$.
\end{enumerate}
\end{proposition}

\begin{proof}
First of all, by Proposition \ref{prop: limit computed by support} we have $\bracket{n , n}_\Omega = \bracket{n , n}_{\Omega^n}$. If $n$ is eventually constant on $\Omega$, it is eventually constant on $\Omega^n$, hence we may assume without loss of generality that $\Omega^n = \Omega$.
\begin{enumerate}

\item If $\Omega$ has finite zeroth and first homology, by Lemma \ref{lem: finite homology equals finite retraction}, $\Omega$ has a finite retraction, call it $\Omega'$. Furthermore, because $n$ is eventually constant on $\Omega$, there exists a finite subquiver $\Omega''$ of $\Omega$ such that $n$ is constant on all arrows that are not in $\Omega''$. Let $\Omega^*$ be a finite IC subquiver which contains $\Omega' \cup \Omega''$. Then $n$ is constant on all arrows not in $\Omega^*$ and $\Omega^*$ is IC and contains a retraction of $\Omega$, hence it is itself a retraction. By Corollary \ref{cor: constant on complement of retraction}, $\bracket{n , n}_{\Omega} = \bracket{n,n}_{\Omega^*}$.

\item If $\Omega$ has finite first homology, it has a finite subquiver containing all cycles. Furthermore there is a finite subquiver such that $n$ is constant on all arrows not contained in it. Hence by Lemma \ref{lem: finite subquivers are contained in finite IC subquivers} there exists a finite IC subquiver $\Omega^*$ containing both which therefore contains all cycles and has the property that $n$ is constant on all arrows outside of it. Then every finite IC subquiver $\Omega'$ containing $\Omega^*$ has $\bracket{n , n}_{\Omega'} \geq \bracket{n , n}_{\Omega^*}$ by Lemma \ref{lem: constant on complement of chain of quivers}, and because $\Omega$ has infinitely many components and $\Omega^*$ only has finitely many, there is an infinite sequence of finite IC subquivers $\Omega^* = \Lambda^1 \subseteq \Lambda^2 \subseteq \ldots$ where each $\Lambda^{p + 1}$ has more components than $\Lambda^p$ and hence again by Lemma \ref{lem: constant on complement of chain of quivers}, $\bracket{n , n}_{\Lambda^p} < \bracket{n , n}_{\Lambda^{p + 1}}$, so $\bracket{n , n}_{\Omega} = +\infty$.


\item If $\Omega$ has finite zeroth homology there is a finite subquiver which contains at least one vertex from every connected component of $\Omega$. Furthermore there is a finite subquiver outside of which $n$ is constant on arrows, hence by Lemma \ref{lem: finite subquivers are contained in finite IC subquivers} there is a finite IC subquiver $\Omega^*$ containing these, which is then BC as well and $n$ is constant on arrows outside of it. Then every finite IC subquiver $\Omega'$ containing $\Omega^*$ has $\bracket{n , n}_{\Omega'} \leq \bracket{n , n}_{\Omega^*}$ by Lemma \ref{lem: constant on complement of chain of quivers}, and because $\Omega$ has infinitely many cycles and $\Omega^*$ only has finitely many, there is an infinite sequence of finite IC subquivers $\Omega^* = \Lambda^1 \subseteq \Lambda^2 \subseteq \ldots$ where each $\Lambda^{p + 1}$ contains a cycle not in $\Lambda^p$ and hence again by Lemma \ref{lem: constant on complement of chain of quivers}, $\bracket{n , n}_{\Lambda^p} > \bracket{n , n}_{\Lambda^{p + 1}}$, so $\bracket{n , n}_{\Omega} = -\infty$. 
\end{enumerate}
\end{proof}

\begin{corollary}
    If $\Omega$ has either finite zeroth or first homology then for any element $n$ of the root space of $\Omega$, the limit $\bracket{n , n}_\Omega$ converges to an element of $\mathbb{Z} \cup \{ \pm \infty \}$. 
\end{corollary}

\begin{proof}
    Given an element $n$ of the root space of $\Omega$, if $\Omega$ has finite first homology then it is clear that $\Omega^n$ does too. On the other hand we claim that if $\Omega$ has finite zeroth homology so does $\Omega^n$. Indeed, if this is not the case then there must be a single component of $\Omega$ which contains infinitely many components of $\Omega^n$. Therefore we may assume without loss of generality that $\Omega$ has one component. Then there exists a finite subquiver $\Omega'$ of $\Omega$ on which $n$ is anticonstant, hence given any vertex $i$ of $\Omega^n$, there is a minimal path from $i$ to a vertex $j$ in $\Omega'$ whose arrows are not contained in $\Omega'$. Hence $n$ is constant along each of these arrows, so the path is in $\Omega^n$. Therefore $i$ is in the same component of $\Omega^n$ as $j$. Since $\Omega'$ is finite, this implies there can only be finitely many components of $\Omega^n$ as desired. 
\end{proof}

\begin{remark}
  If the zeroth and first homology of $\Omega$ are finite, then it follows that the Euler form is well-defined on roots, though we do not need this result for what follows.
\end{remark}

\begin{definition}
  A quiver is called \emph{positive definite} if, for each nonzero element $n$ of the root space for which the Tits form converges, $\bracket{n , n}_{\Omega} > 0$ or $+\infty$.
\end{definition}

\begin{lemma}\label{lm:Tits positive definite}
If $\Omega$ is a quiver  the Tits form on $\Omega$ is positive definite  if and only if the Tits form on every finite IC subquiver of $\Omega$ is positive definite.
\end{lemma}

\begin{proof}
Suppose that the Tits form is positive definite on every finite IC subquiver of $\Omega$ and let $n$ be a nonzero element of the root space of $\Omega$. Then for any finite IC subquiver $\Omega'$ of $\Omega$ we have $\bracket{n , n}_{\Omega'} \geq 0$ with the inequality strict if $\Omega' \cap \Omega^n \neq \emptyset$. Since $\bracket{n , n}_\Omega$ is the limit of non-negative integers, it must be non-negative itself if it exists. Suppose for contradiction that $\bracket{n , n}_\Omega = 0$. Then by definition there exists $\Omega' \in \mathcal{I}(\Omega)$ such that $\bracket{n , n}_{\Omega''} = 0$ for all $\Omega'' \in \mathcal{I}(\Omega)$ containing $\Omega'$. But since $\Omega^n$ is nonempty, by Lemma \ref{lem: finite subquivers are contained in finite IC subquivers} there exists a finite IC subquiver $\Omega''$ of $\Omega$ containing $\Omega'$ and a single vertex from $\Omega^n$. Then $\bracket{n , n}_{\Omega''} \neq 0$ by hypothesis, contradiction.

Conversely, suppose that the Tits form on $\Omega$ is positive definite. If $\Omega'$ is a finite IC subquiver which does not have a positive-definite Tits form, there exists an element $n \neq 0$ of the root space of $\Omega'$ such that $\bracket{n , n}_{\Omega'} \leq 0$. Since $\Omega'$ is finite, it only intersects finitely many components of $\Omega$. Let $\Omega''$ be the union of these components. We claim that $\Omega''$ must have finite first homology. Indeed if not, then one component $\Omega^\dagger$ of $\Omega''$ (and thus of $\Omega$) has infinite first homology, and let $n$ be the element of the root space of $\Omega$ defined by $n_i = 1$ for all $i \in \Omega^\dagger_0$ and $n_i = 0$ for all other $i$. Then $\Omega^n = \Omega^\dagger$, hence by Proposition \ref{lpr: convergence of roots based on supports}, we have $\bracket{n , n}_\Omega = -\infty$, contradiction. Thus $\Omega''$ has finite zeroth and first homology, so by Lemma \ref{lem: finite homology equals finite retraction}, there is a finite retraction $\Omega^*$ of $\Omega''$ which we can assume contains $\Omega'$.


Extend $n$ to $\Omega^*$ by defining it to be zero outside of $\Omega'$. Then we have $\bracket{n , n}_{\Omega^*} = \bracket{n , n}_{\Omega'}$ by Lemma \ref{lem: tits form computed on support finite}. Because $\Omega^*$ is a retraction, each vertex $i$ of $\Omega''$ is connected to $\Omega^*$ by a unique minimal path ending at a vertex $j$. Extend $n$ to $\Omega''$ by letting $n_i = n_j$. Then because $\Omega^*$ is finite, $n$ is constant along all but finitely many arrows so it is an element of the root space of $\Omega''$, and furthermore we have $\bracket{n , n}_{\Omega''} = \bracket{n , n}_{\Omega^*}$ by Lemma \ref{cor: constant on complement of retraction}. Finally extend $n$ to all of $\Omega$ by letting it be zero on all vertices not in $\Omega''$. This is clearly still an element of the root space of $\Omega$, and since the support of $n$ is contained in $\Omega''$ by Proposition \ref{prop: limit computed by support} we have $\bracket{n , n}_{\Omega} = \bracket{n , n}_{\Omega''}$. 

Hence all together we have a nonzero element $n$ of the root space of $\Omega$ with the property that $\bracket{n , n}_{\Omega} = \bracket{n , n}_{\Omega'} \leq 0$, contradiction.



\end{proof}

\begin{proposition}\label{prop: positive tits form if and only if infinite dynkin}
The Tits form is positive definite on the root space of $\Omega$ if and only if the underlying graph of every component of $\Omega$ is of the form  $A_n$, $D_n$, $E_6$, $E_7$, or $E_8$,  $A_\infty$, $A_{\infty, \infty}$, or $D_\infty$).     
\end{proposition}

\begin{proof}
It is positive definite if and only if every component is positive definite.  If a component of $\Omega$ is finite then we know classically it has one of the finite graphs listed. If it is infinite, then it cannot contain any loops or it would have a finite subgraph which is not on that list and therefore not positive definite.  It cannot contain any quadravalent graph, or two tri-valent graphs, or a tri-valent graph with two of the journeys coming out of it longer than $1$.  The only possiblilites left are $D_\infty$ if it contains a trivalent vertex, or $A_\infty$ or $A_{\infty, \infty}$ if it does not. 
\end{proof}

\begin{definition}
    We say that  quiver $\Omega$ is of generalized \emph{ADE  Dynkin type} if $\langle \cdot , \cdot \rangle_\Omega$ is positive definite on the root space. 
\end{definition}

Observe that every FLEI representation of a quiver determines a natural number-valued root called its \emph{dimension vector} that assigns to each vertex the dimension of the associated vector space. We say that a root $n$ of $\Omega$ is a \emph{positive root} if each $n_i \geq 0$ and $\bracket{n,n}_{\Omega} = 1$. 

\begin{proposition} \label{pr: positive roots}
If $\Omega$ is of generalized ADE Dynkin type and eventually outward, then every indecomposable representation has dimension vector a positive root, and every positive root of $\Omega$ is the dimension vector of an indecomposable FLEI representation unique up to isomorphism.
\end{proposition}

\begin{proof}
  Let $V$ be a FLEI representation, let $n$ be its dimension vector (necessarily a root with nonnegative integer entries) and let $\Omega^*$ be a finite retraction such that all arrows outside it are mapped to isomorphisms, and consider the functor $\FF.$  By Lemma~\ref{lem:indecomposables are the same } and the fact that $\Omega^*$ is a retraction $\bracket{n,n}_{\Omega}=\bracket{n,n}_{\Omega^*}$.   By Lemma~\ref{lem:equivalence of categories}  $\FF_{\Omega^*}$ preserves indecomposability, and thus since $\FF_{\Omega^*}(V)$ is indecomposable if and only if its dimension vector has Tits form $1$ and is uniquely determined by its dimension vector, $V$ is indecomposable if and only if its dimension vector has Tits form $1$ and is uniquely determined by its dimension vector.
\end{proof}


\section{Gabriel's Theorem for Infinite Quivers} \label{sec: infininte gabs theorem}

Recall that the represention theory of a quiver is of \emph{unique type} if indecomposables with the same dimension vector are isomorphic and  is \emph{infinite Krull-Schmidt} if every representation is uniquely a (possibly infinite) direct sum of indecomposables. 




\begin{theorem}[Gabriel's Theorem for Infinite Quivers]\label{thm: infinite dimensional infinite gabriel's theorem }
    The category $\text{Reps}(\Omega)$ of a connected quiver $\Omega$ is of unique type and infinite Krull-Schmidt if and only if it is eventually outward and of generalized ADE Dynkin type. In this case all indecomposables of $\Omega$ are FLEI and in one-to-one correspondence with positive roots.  
\end{theorem}

\begin{proof}
Suppose that $\Omega$ is eventually outward and of generalized ADE Dynkin type. If $\Omega$ has underlying graph $A_n$, $D_n$, $E_6$, $E_7$, or $E_8$ then $\Omega$ is Krull-Schmidt  with a one-to-one correspondence between isomorphism classes of indecomposables and positive roots by Gabriel's theorem (see \cite{gabriel1972}) and it is completely reducible by \cite{ringel2016}. (Note that it is obvious that every finite-dimension representation of $\Omega$ is a finite direct sum of indecomposables, but here Ringel proved that any (possibly infinite-dimensional) representation of such an $\Omega$ is a (possibly infinite) direct sum of indecomposables.) Furthermore Gabriel's theorem gives the one-to-one correspondence between the positive roots and the indecomposables. If $\Omega$ has underlying graph $A_\infty$ or $A_{\infty, \infty}$, then the results of \cite{gallup2022decompositions} show that $\text{Reps}(\Omega)$ is of unique type, is infinite Krull-Schmidt, and its indecomposables are in one-to-one correspondence with the connected subquivers, which are clearly in bijection with the positive roots in this case. If $\Omega$ has underlying graph $D_\infty$, then the corresponding results are given by Theorem \ref{thm: decomp of D infinity reps} and Proposition \ref{pr: positive roots}. 

Conversely, suppose that $\Omega$ is not eventually outward. Then there is a journey in $\Omega$ beginning at a vertex $i$ which has infinitely many arrows pointing toward $i$. By assigning all vertices not on this journey the zero vector space, all vertices at the target of an an inward-pointing arrow along this journey according to the representation described in Section 8 of \cite{gallup2022decompositions}, and all outward-pointing arrows along this journey to be isomorphisms, we obtain a representation of $\Omega$ which is not completely reducible and hence the category is not Krull-Schmidt (see Corollary 8.2 of \cite{gallup2022decompositions}).

Finally if $\Omega$ is not of generalized ADE Dynkin type, then by Proposition \ref{prop: positive tits form if and only if infinite dynkin} the Tits form is not positive definite, so by Lemma \ref{lm:Tits positive definite} there is a finite subquiver $\Omega'$ on which the Tits form is not positive definite, so by the proof of Gabriel's theorem (see e.g. \cite{brion2008representations}), there are nonisomorphic indecomposable representations of  $\Omega'$ with the same dimension vector, i.e. $\text{Reps}(\Omega')$ is not of unique type. Extending this dimension vector by zero outside of $\Omega'$ and extending the corresponding indecomposables by zero as well gives that $\text{Reps}(\Omega)$ is not of unique type either. 
\end{proof}


\section{Future Work} \label{sec: future work}

Theorem \ref{thm: infinite dimensional infinite gabriel's theorem } gives one possible generalization of Gabriel's theorem to infinite quivers, but a second option restricts to the category of locally finite-dimensional representations. We conjecture that in this case the ``eventually outward'' condition is not necessary. Precisely, we say that a quiver is \emph{locally finitely reducible} if every locally finite-dimensional representation of $\Omega$ is a direct sum of indecomposables. Then we make the following conjecture. 

\begin{conjecture}[Locally Finite-Dimensional Infinite Gabriel's Thm.]
For a connected quiver $\Omega$ the category $\text{Rep}_{\text{lf}}(\Omega)$ is of unique type and is infinite Krull-Schmit if and only if it is of generalized ADE Dynkin type. In this case all indecomposables of $\Omega$ are FLEI and in one-to-one correspondence with positive roots.
\end{conjecture}

\bibliographystyle{alpha}

\end{document}